\documentclass[11pt,oneside,leqno]{amsart}
\usepackage[utf8]{inputenc}
\usepackage[english]{babel}
\usepackage{amssymb,amsmath,latexsym,amsthm,amsfonts,bbm,dsfont}
\usepackage{geometry}
\usepackage{float}
\usepackage[colorlinks=true,citecolor=red,linkcolor=blue,pdfpagetransition=Blinds]{hyperref}
\numberwithin{equation}{section}
\usepackage{mathtools, stmaryrd}
\usepackage{natbib}
\usepackage{textgreek}
\usepackage{xcolor}
\numberwithin{equation}{section}
\usepackage{xcolor}
\usepackage{graphicx}

\newcommand{\E}{\mathbb{E}}

\newtheorem{theorem}{Theorem}[section]
\newtheorem{proposition}[theorem]{Proposition}
\newtheorem{lemma}[theorem]{Lemma}

\newtheorem{remark}[theorem]{Remark}
\newtheorem{assumption}[theorem]{Assumption}
\newtheorem{example}[theorem]{Example}
\begin{document}

\title[Inverse problems for semi-discrete stochastic parabolic  operators]{Inverse Random Source and Cauchy Problems for Semi-Discrete Stochastic Parabolic Equations in Arbitrary Dimensions}
\author[R. Lecaros]{Rodrigo Lecaros}
\address[R. Lecaros]{Departamento de Matem\'atica, Universidad T\'ecnica Federico Santa Mar\'ia,  Santiago, Chile.}
\email{rodrigo.lecaros@usm.cl}
\author[A. A. P\'erez]{Ariel A. P\'erez}
\address[A. A. P\'erez]{Departamento de Matem\'atica, Universidad del B\'io-B\'io, Concepci\'on, Chile.}
\email{aaperez@ubiobio.cl}
\author[M. F. Prado]{Manuel F. Prado}
\address[M. F. Prado]{(Corresponding Author)Departamento de Matem\'atica, Universidad T\'ecnica Federico Santa Mar\'ia,  Santiago, Chile.}
\email{mprado@usm.cl}
\begin{abstract}
In this paper, we study two types of inverse problems for space semi-discrete stochastic parabolic equations in arbitrary dimensions. The first problem concerns a semi-discrete inverse source problem, which involves determining the random source term of the white noise in the semi-discrete stochastic parabolic equation using observation data of the solution at the terminal time and on an arbitrary open spatial subdomain over a time interval. The second problem addresses a semi-discrete Cauchy inverse problem, which involves determining the solution of the stochastic parabolic equation in a space-time subdomain, from measurements of the solution and the trace of its discrete spatial derivative on an arbitrary open subset of the lateral boundary over a time interval. The key tools are three new global Carleman estimates for the semi-discrete stochastic parabolic operator, one for interior observations and two for boundary observations (homogeneous and nonhomogeneous Dirichlet conditions). Applying these Carleman estimates, we obtain Lipschitz and Hölder stability for the first and second inverse problems, respectively.
\end{abstract}
\maketitle
\tableofcontents 
\section{Introduction}

Let $(\Omega,\mathcal{F},\{\mathcal{F}_{t}\}_{t\geq 0},\mathbb{P})$ be a complete filtered probability space on which a one-dimensional standard Brownian motion $\{B(t)\}_{t\geq 0}$ is defined such that $\{\mathcal{F}_t\}_{t\geq0}$ is the natural filtration generated by $B(t)$, augmented by all $\mathbb{P}$ null sets in $\mathcal{F}$ and $\mathcal{X}$ a Banach space. The space $L^2_{\mathcal{F}}(\Omega;L^2(0,T;\mathcal{X}))$ consists of all $\mathcal{X}$-valued $\{\mathcal{F}_t\}_{t\geq 0}$-adapted processes $X(\cdot)$ such that 
$
\E\left(\int_0^T\|X(\cdot)\|^2_{\mathcal{X}}\,dt\right)<\infty.
$ {  The space $L^2_{\mathcal{F}}(\Omega;H^1(0,T;\mathcal{X}))$ consists of all $\mathcal{X}$-valued $\{\mathcal{F}_t\}_{t\geq 0}$-adapted processes $X(\cdot)$ such that 
$
\E\left(\int_0^T\|X_t(\cdot)\|^2_{\mathcal{X}}+\|X(\cdot)\|^2_{\mathcal{X}}\,dt\right)<\infty.
$ } 
The space $L^2_{\mathcal{F}}(\Omega;L^\infty(0,T;\mathcal{X}))$ consists of all $\mathcal{X}$-valued $\{\mathcal{F}_{t}\}_{t\geq 0}$-adapted bounded processes. Finally, $L^2_{\mathcal{F}}(\Omega;C([0,T];\mathcal{X}))$ denotes the space of all { $\mathcal{X}$}-valued $\{\mathcal{F}_{t}\}_{t\geq 0}$-adapted continuous processes $X(\cdot)$ such that 
$
\E\left(\sup_{t\in[0,T]}\|X(\cdot)\|^2_{\mathcal{X}}\right)<\infty.
$ 
For simplicity, we often write $L^2_{\mathcal{F}}(0,T;\mathcal{X})$, {  $H^1_{\mathcal{F}}(0,T;\mathcal{X})$} and $L^\infty_{\mathcal{F}}(0,T;\mathcal{X})$ whenever no confusion may arise.

Let $T>0$ and $n\in \mathds{N}$, and $G:=(0,1)^{n}\subset \mathbb{R}^n$. We consider the following stochastic parabolic equation: 
\begin{equation}\label{systemofcontrol}
dy-\sum_{i=1}^{n}(\gamma^{i} y_{x_i})_{x_i}dt =\left(\sum_{i=1}^{n}a^{1i}y_{x_i}+a^2y\right)dt+(a^3\,y+g)dB(t) \quad \text{in } G\times (0,T).
\end{equation}
In \eqref{systemofcontrol}, $y$ denotes the state variable, $g$ is a source term representing the intensity of the random force of white noise, and $\gamma^{i}, a^{1i}$, and $a^{k}$ (for $i=1,\ldots,n$ and $k=2,3$) are given functions. We refer to \cite{KR77}, \cite{Zho92}, and \cite{lu2021mathematical} for the existence and uniqueness results of \eqref{systemofcontrol}.

The main objective of this paper is to study the inverse random source problem and the Cauchy problem associated with the semi-discrete spatial approximation of the stochastic parabolic equation \eqref{systemofcontrol}. To this end, we introduce the spatial discretization employed throughout this work. 

Let $N\in\mathbb{N}$ be sufficiently large and define a set of points $\mathcal{M}= G\cap h\mathbb{Z}^{n}$ with $h=1/(N+1)$, called the primal mesh, we define the dual mesh in the direction $e_{k}$, where $\{e_{k}\}_{k=1}^{n}$ denotes the standard basis of $\mathbb{R}^{n}$, by
\begin{equation*}
    \mathcal{M}_{k}^{\ast}:= \left\{ x\pm\frac{h}{2}e_{k}\mid x\in \mathcal{M}\right\}.
\end{equation*}
This allows us to define the boundary points in the $e_{k}$ direction as $
    \partial_{k}\mathcal{M}:= (\mathcal{M}^{\ast}_{k})^{\ast}_{k}\setminus \mathcal{M}$,
and we set
\[
        \partial \mathcal{M}:=\bigcup_{k=1}^{n} \partial_{k}\mathcal{M},\quad\textrm{ and }\quad 
        \overline{\mathcal{M}}:=\bigcup_{k=1}^{n} (\mathcal{M}^{\ast}_{k})^{\ast}_{k}.
\]
We denote by $C(\mathcal{M})$ the set of real-valued functions defined in the mesh $\mathcal{M}$. 

The outward normal to $\mathcal{M}$ in the direction $e_{i}$ is defined as $\nu_{i}\in C(\partial\mathcal{M}_{i})$:
\begin{equation}\label{normal:function}
    \forall x\in\partial_{i}\mathcal{M}, \quad
    \nu_{i}(x):=\begin{cases} 
     1 &  \mbox{if }x-\tfrac{h}{2}e_i\in \mathcal{M}_{i}^{\ast} \ \mbox{and } x+\tfrac{h}{2}e_i\notin\mathcal{M}_{i}^{\ast},\\ 
    -1 &  \mbox{if }x-\tfrac{h}{2}e_i\notin \mathcal{M}_{i}^{\ast} \ \mbox{and }x+\tfrac{h}{2}e_i\in\mathcal{M}_{i}^{\ast},\\ 
     0 &  \mbox{otherwise}.
    \end{cases}
\end{equation}
We also define the trace operator $t_{r}^{i}$ for $u\in C(\mathcal{M}^{\ast}_{i})$ by
\begin{equation}
    \forall x\in \partial_{i}\mathcal{M},\quad 
    t_{r}^{i}(u)(x):=\begin{cases}
    u\!\left(x-\tfrac{h}{2}e_i\right) & \nu_{i}(x)=1,\\
    u\!\left(x+\tfrac{h}{2}e_i\right) & \nu_{i}(x)=-1,\\
     0 & \nu_{i}(x)=0 .
    \end{cases}
\end{equation}
Finally, to obtain the discrete version of our system via the finite difference method, we introduce the difference and average operators in the direction $e_{k}$:
\begin{equation*}
    (D_{k}u)(x):=\frac{\tau_{+k}u(x)-\tau_{-k}u(x)}{h},\quad
    (A_{k}u)(x):=\frac{\tau_{+k}u(x)+\tau_{-k}u(x)}{2},\quad \text{for }x\in\mathcal{M}^{\ast}_{k},
\end{equation*}
where $\tau_{\pm k}u(x):= u\!\left(x\pm\frac{h}{2}e_{k}\right)$.

The Banach spaces $\mathcal{X}$ required for the formulation of both inverse problems are defined as follows. We denote by $L^p_{h}(\mathcal{M})$ with $1\leq p<\infty$ the space of function defined on $\mathcal{M}$ with the associated norm.
\[
\|u\|_{L_{h}^p(\mathcal{M})}^p:= \int_{\mathcal{M}}|u|^p\quad \text{with}\quad
\int_{\mathcal{M}}|u|^p:= h^n \sum_{x\in \mathcal{M}}|u(x)|^p.
\]
Analogously, we denote by $W^{1,p}_h(\mathcal{M})$, $W_h^{2,p}(\mathcal{M})$, and $L^{\infty}_{h}(\mathcal{M})$ the set of functions defined in $\mathcal{M}$ with the associated norms:
\begin{equation*}
    \|u\|_{W_h^{1,p}(\mathcal{M})}^p := \|u\|_{L_h^p(\mathcal{M})}^p+{ \sum_{i=1}^{n}}\int_{\mathcal{M}_i^{\ast}}|D_iu|^p,
\end{equation*}
\begin{equation*}
    \|u\|_{W_h^{2,p}(\mathcal{M})}^p :=  \|u\|_{W_h^{1,p}(\mathcal{M})}^p+{ \sum_{i,j=1}^{n}}\int_{{ \mathcal{M}_{ij}^{\ast}}}|D_{ij}^2u|^p,\quad\text{and}\quad \|u\|_{L_h^{\infty}(\mathcal{M})}:=\max_{x\in \mathcal{M}}|u(x)|,
\end{equation*}
respectively. In particular, for $p=2$, we adopt the usual notation $W^{1,2}_h(\mathcal{M})=H_h^{1}(\mathcal{M})$ and $W_h^{2,2}(\mathcal{M})=H_h^{2}(\mathcal{M})$.  

For the boundary, we proceed as above, that is, for $u\in C(\partial \mathcal{M})$ we define the integral over the boundary $\partial\mathcal{M}$ as
\[
\int_{\partial \mathcal{M}}u:= \sum_{i=1}^{n}\int_{\partial_i\mathcal{M}}u(x)\quad \text{with}\quad
\int_{\partial_i\mathcal{M}}u(x):= h^{n-1}\sum_{x\in\partial_i\mathcal{M}} u(x).
\]
Then, we denote as $L^{2}_{h}(\partial\mathcal{M})$ and $H^{1}_{h}(\partial\mathcal{M})$ the set of functions defined on $\partial\mathcal{M}$ with the associated norm: \[
\left\|u \right\|^{2}_{L^{2}_{h}(\partial\mathcal{M})}:=\int_{\partial\mathcal{M}}|u|^2\quad \text{and}\quad\left\| u\right\|^{2}_{H^{1}_{h}(\partial\mathcal{M})}:={ \sum_{i=1}^n}\left(\int_{\partial_i \mathcal{M}}|u|^2+\sum_{\substack{j=1 \\ j \neq i}}^{n}\int_{(\partial_i\mathcal{M})_{j}^{\ast}}|D_ju|^2\right),
\]
respectively. Finally, following \cite{EDG:2011}, we adopt the discrete $H^{1/2}_{h}$-norm defined for $v \in C(\partial\mathcal{M})$ by
\[
\|v\|_{H_{h}^{1/2}(\partial \mathcal{M})}:= \min_{\left.u\right|_{\partial\mathcal{M}}=v}\|u\|_{H^1_{h}(\mathcal{M})}.
\]

Following the same reasoning, we analogously define the semi-discrete spaces \\
$L^2_{\mathcal{F}}(0,T;L_h^\infty(\mathcal{M}))$, $L^2_{\mathcal{F}}(0,T;H_h^1(\mathcal{M}))$, $L^2_{\mathcal{F}}(0,T;L_h^2(\partial\mathcal{M}))$, $L^2_{\mathcal{F}}(0,T;H_h^1(\partial\mathcal{M}))$, and \\
$L^2_{\mathcal{F}}(0,T;H_h^{1/2}(\partial\mathcal{M}))$.

Given the preceding definitions, we are now ready to formulate the problems to be addressed. Let $\gamma_i$ denote the sampling of $\gamma^{i}$ on both the primal and dual meshes. Furthermore, { let $a_{1i}$, $a_2$, and $a_3$ represent the evaluations of $a^{1i}$, $a^2$, and $a^3$} on the primal mesh, for $i = 1, \ldots, n$, respectively. We now formulate the semi-discrete counterpart of the inverse problems considered in this paper, namely, the semi-discrete inverse random source problem and the semi-discrete Cauchy problem for $n$-dimensional stochastic parabolic equation.

\textbf{Semi-discrete inverse random source problem.} We consider the following $n$-dimensional semi-discrete stochastic parabolic equation, derived from \eqref{systemofcontrol}, assuming $a_3 = 0$:
\begin{equation}\label{systemrandomsource}
    \begin{cases}
    \mathcal{A}_{h}w=\left(\sum_{i=1}^n a_{1i}A_iD_i(w)+a_2w\right)dt+gdB(t),     &\text{in } Q:= \mathcal{M}\times (0,T),  \\
    w=0,\quad\text{on } { \partial Q:=\partial \mathcal{M}\times (0,T)},&\left.w\right|_{t=0}=w_0, \quad\text{in } \mathcal{M},
    \end{cases}
\end{equation}
where the second-order operator is defined by $\mathcal{A}_h w := dw - \sum_{i=1}^n D_i(\gamma_i D_i(w))dt$, and $w_0$ is the initial condition. 

The following assumption on $\gamma_i$ will be needed throughout the paper. Let $\gamma_i>0$ for all $i=1,...,n$ and 
$$
 \mbox{reg}(\gamma):= \text{ess}\sup_{\substack{x\in G \\ i=1,\cdots,n}}\left(\gamma_i+\frac{1}{\gamma_i}+{ \sum_{j=1}^{n}|\partial_{j}\gamma_i|^2}\right)<c_0.
$$
for any constant $c_0> 0$.
{ 
Now, we define the observation operator $\Lambda_1: L^2_{\mathcal{F}}(0,T;H_h^1(\mathcal{M})) \longrightarrow X_1
$
as
$$
\Lambda_1(g) := \left(\left.w\right|_{(G_0\cap\mathcal{M})\times(0,T)}, \left.w\right|_{t=T}\right),
$$
where $G_0$ is an arbitrary open subset of $G$, and $w$ denotes the solution to \eqref{systemrandomsource} associated with the source term $g$.} We set $X_1:= L^2_{\mathcal{F}}(0,T;L_h^2(G_0\cap\mathcal{M}))\times L^2_{\mathcal{F}_T}(\Omega;L^2_{h}(\mathcal{M}))$ endowed with the norm
$$
\|(f_1,f_2)\|_{X_1}:= \|f_1\|_{L^2_{\mathcal{F}}(0,T;L_h^2(G_0\cap\mathcal{M}))}+\|f_2\|_{L^2_{\mathcal{F}_T}(\Omega;L^2_{h}(\mathcal{M}))}.
$$
The direct problem consists in solving \eqref{systemrandomsource} for a given $g$ and then computing $\Lambda_{1}(g)$. The inverse problem is to recover $g$ from the observation $\Lambda_1(g)$. We now state our first main result, which provides a conditional stability estimate for this inverse problem.

\begin{theorem}\label{firstresulttheorem}
Let $a_{1i}\in L^\infty_{\mathcal{F}}(0,T;L_h^\infty(\mathcal{M}))$ for $i=1,\ldots,n$, and $a_2\in L_{\mathcal{F}}^{\infty}(0,T;L_h^{n^{\ast}}(\mathcal{M}))$ where $n^{\ast}$ satisfies
\begin{equation}\label{conditionofn^{ast}}
\begin{cases}
    n^{\ast}>2,&\quad n=2,\\
    n^{\ast}\geq n,&\quad n\geq 3.
\end{cases}    
\end{equation}
Assume that $g_k \in L^2_{\mathcal{F}}(0,T;H_h^1(\mathcal{M}))$ for $k=1,2$, and that, for $i=1,\ldots,n$,
\begin{equation}\label{conditionong}
|D_i(g_1 - g_2)| \leq \mathcal{C}|A_i(g_1 - g_2)|,  \text{ in } { Q_i^*:=(0,T)\times\mathcal{M}_{i}^{\ast}}, \; \mathbb{P}\text{-a.s.}
\end{equation}
Then, there exists a constant $\mathcal{C} > 0$, independent of $h$, such that
\begin{equation}\label{eq:firstresult}
\|g_1 - g_2\|_{L^2_{\mathcal{F}}(0,T;L_h^2(\mathcal{M}))} \leq \mathcal{C} \|\Lambda_1(g_1) - \Lambda_1(g_2)\|_{X_1}.
\end{equation}
\end{theorem}

{ \begin{remark}
    Here we provide an example that satisfies \eqref{conditionong}. Indeed, \begin{equation*}
    g(x) = \prod_{k=1}^n (x_k + C')\mathcal{S}(t,\omega),
\end{equation*}
for any constant $C' > 0$ and $\mathcal{S}\in L^{2}_{\mathcal{F}}(0,T);\mathbb{R})$ is a real-valued adapted process satisfying $\mathcal{S}(t,\omega)>0$ $\mathbb{P}\text{-a.s.}$. To verify condition \eqref{conditionong}, fix a direction
$i \in \{1, \ldots, n\}$ and write $g(x) = (x_{i} + C') \cdot R_{i}(x)$, where
\begin{equation*}
    R_{i}(x) := \prod_{k \neq i} (x_{k} + C')\mathcal{S}(t,\omega).
\end{equation*}
Since the discrete operators $D_{i}$ and $A_{i}$ act only on the $i$-th variable, a direct
computation on $\mathcal{M}_{i}^{\ast}$ gives
\begin{equation*}
    D_{i} g(x)
    = \frac{g\!\left(x + \tfrac{h}{2}e_{i}\right) - g\!\left(x - \tfrac{h}{2}e_{i}\right)}{h}
    = \frac{\bigl(x_{i} + \tfrac{h}{2} + C'\bigr) - \bigl(x_{i} - \tfrac{h}{2} + C'\bigr)}{h}
      \cdot R_{i}(x)
    = R_{i}(x),
\end{equation*}
\begin{equation*}
    A_{i} g(x)
    = \frac{g\!\left(x + \tfrac{h}{2}e_{i}\right) + g\!\left(x - \tfrac{h}{2}e_{i}\right)}{2}
    = (x_{i} + C') \cdot R_{i}(x).
\end{equation*}
Therefore,
\begin{equation*}
    \frac{|D_{i} g(x)|}{|A_{i} g(x)|}
    = \frac{R_{i}(x)}{(x_{i} + C')\, R_{i}(x)}
    = \frac{1}{x_{i} + C'}
    \leq \frac{1}{C'},
    \qquad \forall\, x \in \mathcal{M}_{i}^{\ast},
\end{equation*}
where we used $x_{i} \geq 0$ and $\mathcal{S}(t,\omega)>0$ $\mathbb{P}\text{-a.s.}$. Hence condition (1.6) holds with constant
$C = 1/C'$, uniformly in $h$, for every $i = 1, \ldots, n$.
  More generally, it is possible to build a family of functions that verify that condition. For any adapted process $g\in L^{2}_{\mathcal{F}}(0,T;H_{h}^{1}(\mathcal{M}))$ satisfying  $\|D_{i}g\|_{L^{\infty}_{\mathcal{F}}(0,T;L^{\infty}_{h}(\mathcal{M}_{i}))}<C'$ and $|A_{i}g(x,t,\omega)|\geq\alpha>0$ for all $x\in \mathcal{M}_{i}$, $t\in (0,T)$, $\mathbb{P}\text{-a.s.}$; it follows that $|D_{i}g(x,t,\omega)|\leq C'=\frac{C'}{\alpha}\cdot\alpha\leq C|A_{i}g(x,t,\omega)|$ with $C=\tfrac{C'}{\alpha}$.
\end{remark}}

\textbf{Semi-discrete Cauchy Problem.} We consider the same spatial semi-discretization of \eqref{systemofcontrol}, with $g = 0$:{ 

\begin{equation}\label{eq:cauchy problem}
    \mathcal{A}_{h}w=\left(\sum_{i=1}^n a_{1i}A_iD_i(w)+a_2w\right)dt+a_3w\,dB(t),\quad \text{in } Q,
\end{equation}
}
While the discretization shares structural similarities with the previous inverse problem, the nature of the available data differs. In this setting, our objective is to recover $w$ within a subdomain $D \subset Q$. This is achieved through the observation operator $\Lambda_2: L^2_{\mathcal{F}}(0,T;H_h^2(\mathcal{M})) \to X_2 $, which is defined as follows:
\[
\Lambda_2(w) := \left(\left.w\right|_{\Gamma \times (0,T)},\left. \partial_\nu w\right|_{\Gamma \times (0,T)}\right),
\]
where $w$ solves \eqref{eq:cauchy problem}. The normal derivative is given by
\[
\partial_{\nu}w := \sum_{i=1}^n t_r^{i}(D_i w)e_i\nu_i,
\]
 and the set $\Gamma=\Gamma_0\cap\partial M$, where $\Gamma_0$ is an arbitrary non-empty open subset of $\partial G$. In this case, we set $X_2:= H^1_{\mathcal{F}}(0,T;H_h^1(\Gamma) \cap H_h^{1/2}(\Gamma)) \times \left(L^2_{\mathcal{F}}(0,T;L_h^2(\Gamma))\right)^n$ endorsed with norm $$
\left\|\left(f_1,f_2,...,f_{n},f_{n+1}\right)\right\|_{X_2}:= \|f_{1}\|_{H^1_{\mathcal{F}}(0,T;H_h^1(\Gamma) \cap H_h^{1/2}(\Gamma)) }+\sum_{i=2}^{n+1}\|f_i\|_{L^2_{\mathcal{F}}(0,T;L_h^2(\Gamma))},
$$
for $f_{1}\in H^1_{\mathcal{F}}(0,T;H_h^1(\Gamma) \cap H_h^{1/2}(\Gamma))$ and $f_i\in L^2_{\mathcal{F}}(0,T;L_h^2(\Gamma))$ with $i=2,\ldots,n+1$.

The inverse problem consists of determining the solution $w$ in an open subdomain $\mathcal{M}_0 \subset \mathcal{M}$ such that $\overline{\mathcal{M}}_0 \subset \mathcal{M} \cup \Gamma$ and $\partial \mathcal{M}_0 \cap \partial \mathcal{M} \subsetneq \Gamma$ over the time interval $(\varepsilon, T - \varepsilon)$, based on the boundary measurement given by $\Lambda_2(w)$. We establish the following Hölder-type stability estimate:

\begin{theorem}\label{Teo:SecondResult}
 Let $a_{1i},a_2\in L^\infty_{\mathcal{F}}(0,T;L_h^\infty(\mathcal{M}))$, $a_3 \in L^\infty_{\mathcal{F}}(0,T;W_h^{1,\infty}(\mathcal{M}))$, \break$\xi \in H^1_{\mathcal{F}}(0,T;H_h^1(\Gamma)\cap H_{h}^{1/2}(\Gamma))$, and let $\varepsilon > 0$. Then there exist constants $\mathcal{C} > 0$, $h^* > 0$, and $\kappa \in (0,1)$ such that for all $h \in (0, h^*)$ and any solution $w \in L^2_{\mathcal{F}}(0,T;H_h^2(\mathcal{M}))$ of \eqref{eq:cauchy problem} with $w=\xi\; \text{on}\;\, \Gamma\times (0,T)$ satisfying
\[
\|w\|_{L^2_{\mathcal{F}}(0,T;H_h^2(\mathcal{M}))} \leq \mathbf{M},
\]
the following estimate holds
\begin{equation}\label{eq:teo:secondresult}
\|w\|_{L^2_{\mathcal{F}}(\varepsilon,T-\varepsilon;H_h^2(\mathcal{M}_0))} \leq \mathcal{C} \max{\left\{\|\Lambda_2(w)\|_{X_2},\mathbf{M}e^{-\mathcal{C}h^{-1}},\mathbf{M}^{\kappa}\|\Lambda_2(w)\|_{X_2}^{1-\kappa}\right\}}.
\end{equation}
\end{theorem}
{ 
\begin{remark}\label{remark:nonuniquess} We observe that by refining the mesh size $h$ and the observation region such that $\|\Lambda_2(w)\|_{X_2}<\mathbf{M}$,  inequality \eqref{eq:teo:secondresult} can be rewritten as
\begin{equation}\label{ine:generalization}
   \|w\|_{L^2_{\mathcal{F}}(\varepsilon,T-\varepsilon;H_h^2(\mathcal{M}_0))} \leq \mathcal{C} \mathbf{M}^{\kappa}\|\Lambda_2(w)\|_{X_2}^{1-\kappa}.
\end{equation}
The above inequality stands for a generalization of \cite[Theorem 2.6]{WWW:2024}. Moreover, in \cite[Remark 2.11]{WWW:2024} the authors claim that \cite[Theorem 2.6]{WWW:2024} implies the uniqueness of this inverse problem. This latter conclusion is incorrect. Indeed, by examining the proof of that Theorem  and assuming that the measurements are zero, we see that in the unlabeled inequality preceding \cite[Inequality 5.16]{WWW:2024}, an exponential error term remains, which prevents one to concluding a uniqueness result. For this reason, in inequality \eqref{eq:teo:secondresult} appears the term $\mathbf{M}e^{-\mathcal{C}h^{-1}}$. There is another way to see that a uniqueness result is not possible. From \eqref{ine:generalization}, we obtain, up to a constant, $\mathcal{C}e^{-h^{-1}}\leq \|\Lambda_{2}(w)\|^{1-\kappa}_{X_{2}}$. Then, assuming zero measurement does not imply uniqueness in this inverse problem. This phenomenon is not new in discrete inverse problem: in the continuous case, that term vanishes when the Carleman parameter is taken sufficiently large, while in the discrete setting this is not allowed because the discrete mesh size and the Carleman parameter are connected, see Theorems \ref{theo:Carleman_datainterior} and \ref{theo:CarlemanDataboundary}. Moreover, similar results are obtained in several estimates for discrete systems; see for instance \cite{EDG:2011,LOPD:2023}.
 \end{remark}}
The two inverse problems under consideration are associated with stochastic parabolic equations. Stochastic partial differential equations (SPDEs) have broad applications across multiple disciplines. They model phenomena such as particle diffusion and heat propagation in physics, option pricing in finance, epidemic spread and population dynamics in biology, systems with uncertain parameters in engineering, and pollutant dispersion in environmental sciences. These equations incorporate randomness and uncertainty, making them particularly suitable for modeling inherently unpredictable phenomena. Consequently, stochastic PDEs often provide more realistic representations than their deterministic counterparts. Nevertheless, research on inverse problems for stochastic PDEs remains limited. For further details, see \cite[Chapter 5]{lu2021mathematical} and the references therein.

For the inverse source problem related to stochastic parabolic equations, \cite{QLu:2012} establishes uniqueness in determining the source function $f$ in the drift term. The inverse source problem of simultaneously determining two types of sources, $f$ (drift term) and $g$ (diffusion term), in a stochastic parabolic equation was studied in \cite{ref37} using observations at the final time and on the lateral boundary. A unique result for this inverse problem was obtained using a global Carleman estimate for the continuous stochastic parabolic equation. In addition, \cite{ref1} and \cite{ref2} discuss the application of regularization techniques within numerical methods to inverse random source problems.

The Cauchy problem aims to recover the solution using the observed data on the lateral boundary. In \cite{Yamamoto_2009}, conditional stability was established for the Cauchy problem of deterministic parabolic equations. The conditional stability and the convergence rate of the Tikhonov regularization method for the Cauchy problem of stochastic parabolic equations were obtained in \cite{ref13}. However, all of the above cited results were obtained within a continuous framework, as is the survey \cite{LZ:2024}. For semi-discrete inverse problems related to stochastic differential equations, see \cite{WWW:2024}, which studies the one-dimensional case of both issues, including a semi-discrete inverse random source problem and a semi-discrete Cauchy problem. 

This article builds upon the results of \cite{WWW:2024}, extending them to arbitrary spatial dimensions and providing several refinements. In particular, Theorem \ref{firstresulttheorem} provides integrability conditions for the coefficient $a_2$ that depend on the spatial dimension of the problem. Furthermore, Theorem \ref{Teo:SecondResult} establishes a stability estimate in which the right-hand side of inequality \eqref{eq:teo:secondresult} explicitly reveals the dependence on the mesh size $h$ and the size of the observation operator. As a consequence, the second result recovers the same estimate as in \cite{WWW:2024} when $h$ and the observation region are chosen sufficiently small. The main tools to address the previous inverse problems are two Carleman estimates. The first is for a homogeneous Dirichlet condition.
{ 
\begin{theorem}\label{theo:Carleman_datainterior}
Let $\psi$ satisfying assumptions \ref{Assumption_data_inteior} and $\varphi=e^{\lambda\psi}$. For $\lambda_1\geq 1$ sufficiently large, there exist $\mathcal{C}$, $\tau_{0}\geq 1$, $h_{0}>0$, $\varepsilon_0 >0$, depending on $G,\,G_0$, $\beta$, $c_0$, $T$, and $\lambda$, such that

\begin{equation}\label{eq:firstCarlemanV2}
\begin{split}
&J(w)+\E\int_{Q}s^{3}\lambda^{4}_{1}\varphi^{3}\,e^{2s\varphi}|w|^{2}\,dt+\sum_{i,j=1}^{n}\E\int_{Q_{ij}^{\ast}}s^{-1}e^{2s\varphi}|D_{ij}^{2}w|^2\, dt+\E\int_{Q}s\lambda_{1}^2\varphi\,e^{2s\varphi}|g|^{2}dt\\
&\leq\,\mathcal{C}_{\tau_{0},\varepsilon_0}\left(\E\int_{Q}\,e^{2s\varphi}|f|^2\,dt+\sum_{i=1}^{n}\E\int_{Q_{i}^{\ast}}\,e^{2s\varphi}|D_{i}g|^{2}dt+\int_{0}^{T}\int_{G_{0}\cap\mathcal{M}}s^{3}\varphi^{3}e^{2s\varphi}|w|^{2}\,dt\right.\\
&\left.+\left.\E\int_{\mathcal{M}}s^2\,e^{2s\varphi}|w|^{2}\right|_{t=T}\right),
\end{split}
\end{equation}
for all $\tau\geq \tau_{0}$, $0<h\leq h_0$, $s(t)h\leq \epsilon_0$, where $\displaystyle J(w):= \sum_{i=1}^{n}\E\int_{Q_{i}^{\ast}}s\lambda_1^2\varphi e^{2s\varphi}\,|D_{i}w|^{2}\,dt+\E\int_{Q}s\lambda_{1}^{2}\varphi e^{2s\varphi}|A_iD_iw|^{2}\,dt$, $f\in L^{2}_{\mathcal{F}}(0,T;L^{2}_{h}(\mathcal{M}))$, $g\in L^{2}_{\mathcal{F}}(0,T;H^{1}_{h}(\mathcal{M}))$, and $w$ satisfies $\displaystyle dw-\sum_{i=1}^{n} D_i(\gamma_i D_{i}w)\,dt=fdt+gdB(t)$ with $w=0$ on $\partial \mathcal{M}$ and $w|_{t=0}=0$ in $\mathcal{M}$.
\end{theorem}
}
The second is for the nonhomogeneous Dirichlet condition given in Theorem \ref{theo: CarlemanNonhomgeneous.}. { Theorem \ref{theo:CarlemanDataboundary} provides, to our knowledge, the first Carleman estimate for a stochastic forward semi-discrete parabolic operator with interior observation. Moreover, Theorems \ref{theo:Carleman_datainterior}--\ref{theo: CarlemanNonhomgeneous.} extend the one-dimensional Carleman estimates established in \cite{WWW:2024}. We adopt the general structure of the weight function introduced in that work, but modify its spatial component in order to treat more general configurations. This modification allows us to derive Carleman estimates corresponding to interior or boundary observations on arbitrary subsets of the corresponding domains.
}However, in the multidimensional setting, additional difficulties arise due to the needed to incorporate second-order terms $D^{2}_{ij}$ and the fact that the boundary is no longer a simple pair of points. To establish these estimates, we rely on fundamental discrete tools, such as integration by parts and the product rule, which we state below.

\begin{proposition}[{\cite[Lemma 2.2]{LOPD:2023}}]\label{prop:integralbyparts}
For any $v\in C(\mathcal{M}_{i}^{\ast})$ and $u\in C(\overline{\mathcal{M}}_{i})$, the following identities hold: for the difference operator
\begin{equation}\label{eq:int:dif}
    \int_{\mathcal{M}}u\,D_{i}v=-\int_{\mathcal{M}_{i}^{\ast}}v\,D_{i}u+\int_{\partial_{i}\mathcal{M}}u\,t_{r}^{i}(v)\nu_{i},
\end{equation}    
and for the average operator
\begin{equation}\label{eq:int:ave}
\int_{\mathcal{M}}u\,A_{i}v=\int_{\mathcal{M}_{i}^{\ast}}v\,A_{i}u-\frac{h}{2}\int_{\partial_{i} \mathcal{M}}u\,t_{r}^{i}(v).
\end{equation}
\end{proposition}

Next, we recall the discrete product rule, which relates the average and difference operators.

\begin{proposition}[{\cite[Lemma 2.1 and 2.2]{BHLR:2010b}}]\label{pro:product}
For $u,v\in C(\overline{\mathcal{M}})$, the following identities hold in $\mathcal{M}^{\ast}_{i}$. For the difference operator
\begin{equation}\label{eq:difference:product}
    D_{i}(u\,v)=D_{i}u\, A_{i}v+A_{i}u\,D_{i}v,
\end{equation}
and for the average operator 
\begin{equation}\label{eq:average:product}
     A_{i}(u\,v)= A_{i}u\,A_{i}v+\frac{h^{2}}{4}D_{i}u\,D_{i}v.
\end{equation}
Moreover, on ${\mathcal{M}}$ we have
\begin{equation}\label{eq:averengeanddifference}
    u=A^2_i u-\frac{h^2}{4}D^2_iu.
\end{equation}
\end{proposition}
The remainder of this paper is organized as follows. In Section \ref{sec:RandomSource}, we prove a Lipschitz stability result of the spatial semi-discrete inverse random source problem applying the result of the Carleman estimate \eqref{eq:firstCarlemanV2} where its proof is presented in the Appendix \ref{Proof:CarlemanEstimateHomegeneuos}. Finally, in Section \ref{sec:cauchyproblem}, we study the spatial semi-discrete inverse Cauchy problem and establish a Hölder stability result based on a second Carleman estimate and a geometric argument.

\section{Semi-discrete inverse random source problem}\label{sec:RandomSource}
In this section, we establish the stability of the semi-discrete inverse random source problem (Theorem \ref{firstresulttheorem}). 
The key tool for proving this result is a semi-discrete Carleman estimate, whose proof is deferred to {  Appendix \ref{Proof:CarlemanEstimateHomegeneuos}}. To formulate this estimate, { 
the weight function is defined as $r=e^{s\varphi}$ with $\varphi=e^{\lambda \psi(x)}$, for $\lambda>1$ and with $\psi$ satisfies the following assumption. The construction of this weight function is classic (see  e.g \cite{Emanuilov:1995}) and the existence of this function in the semi-discrete setting is given in \cite[Appendix A]{BHLR:2010b}. 
\begin{assumption}\label{Assumption_data_inteior}
    Let $\tilde{G}_0\subsetneq G_0$ be an open set. Let $\tilde{G}$ be a smooth open and connected neighborhood of $\overline{G}$ in $\mathbb{R}^n$. The function $x\mapsto \psi(x)$ in $C^k(\overline{\tilde{G}},\mathbb{R})$, $k$ sufficiently large, and satisfies, for some $c>0$,
\begin{equation*}
\begin{array}{ccc}
    \psi>0\; \text{in}\; \tilde{G},\; \psi=0\; \text{on} \;\partial\tilde{G}, & |\nabla\psi|\geq c\;\; \text{in}\; \tilde{G}\setminus \tilde{G}_0, & \text{and}\; \partial_{\nu_i}\psi(x)\leq -c<0,\;\text{for}\,\; x\in V_{\partial_i G}, 
\end{array}
\end{equation*}
where $V_{\partial_iG}$ is a suffciently small neighborhood of $\partial_iG$ in $\tilde{G}$, in which the outward unit normal $\nu_i$ to $G$ is extended from $\partial_i G$. 
\end{assumption}
On the other hand, for the temporal weight function, we define it as:} for $t_0\in (0,T)$ and $\beta>0$,
\begin{equation}\label{theta-delta}
    \theta(t)=\beta|t-t_0|^2.
\end{equation}
Given $\tau\geq 1$, we set $
s(t)=\tau e^{-\lambda \theta(t)}.
$
\begin{remark}
    Compared to \cite{LPP:2025} and other works on parabolic operators, the weight function here has been modified, particularly in its temporal component; however, the structure $s(t)\varphi(x)$ is preserved. Moreover, the functions satisfy $|\theta(t)| \leq C$ and $|\theta_t(t)| \leq C$ for some constant $C>0$ and for all $t \in [0,T]$. As a result, the estimates in \cite{BHLR:2010b,AA:perez:2024} holds, and most of the cross-product terms can be treated in a manner analogous to \cite{LPP:2025}, slightly simplifying the computations for the new Carleman estimate in this article.
\end{remark}

\subsection{Proof of Theorem \ref{firstresulttheorem}.}\label{Proof:firstresult}

Let $w_{1}$ and $w_{2}$ be the solutions of \eqref{systemrandomsource} with the random forcing term $g_{1}$ and $g_{2}$, respectively. Setting $\tilde{w}=w_{1}-w_{2}$ and $\tilde{g}=g_{1}-g_{2}$, we obtain
\begin{equation*}
    \begin{cases}
     d\tilde{w}-\sum_{i=1}^{n}D_i(\gamma_{i}D_i\tilde{w})=\left({ \sum_{i=1}^{n}}a_{1i}A_iD_i(\tilde{w})+a_2\tilde{w}\right)\,dt+\tilde{g}\,dB(t),    &\text{in}\,Q,  \\
 \tilde{w}=0,\quad\text{on}\, \partial Q,\quad\quad\quad\left.\tilde{w}\right|_{t=0}=0,\quad \text{in}\,\mathcal{M}.
    \end{cases}
\end{equation*}
Applying Theorem \ref{theo:Carleman_datainterior} to $\tilde{w}$,  we derive the following estimate
{ 
\begin{equation*}
\begin{split}
&J(\tilde{w})+\E\int_{Q}s^{3}\lambda^{4}_{1}\varphi^{3}\,e^{2s\varphi}|\tilde{w}|^{2}\,dt+\sum_{i,j=1}^{n}\E\int_{Q_{ij}^{\ast}}s^{-1}e^{2s\varphi}|D_{ij}^2\tilde{w}|^2\, dt+\E\int_{Q}s\lambda_{1}^2\varphi\,e^{2s\varphi}|\tilde{g}|^2dt\\
&\leq\,\mathcal{C}_{\tau_{0},\varepsilon_0}\left(\E\int_{Q}\,e^{2s\varphi}\left|\sum_{i=1}^{n}a_{1i}A_iD_i\tilde{w}+a_2\tilde{w}\right|^2\,dt+\sum_{i=1}^{n}\E\int_{Q_i^{\ast}}\,e^{2s\varphi}|D_i\tilde{g}|^2dt\right.\\
&\left.+\int_{0}^T\int_{G_0\cap\mathcal{M}}s^3\varphi^3e^{2s\varphi}|\tilde{w}|^2\,dt+\left.\E\int_{\mathcal{M}}s^2\,e^{2s\varphi}|\tilde{w}|^2\right|_{t=T}\right).
\end{split}
\end{equation*}
}
By applying Young's inequality for the first integral and using \eqref{conditionong} for the second term on the right-hand side of the previous inequality, we can rewrite the inequality above as
\begin{equation}\label{eq:inequalityEC}
\begin{split}
&J(\tilde{w}) + \E\int_{Q}s^{3}\lambda^{4}_{1}\varphi^{3}\,e^{2s\varphi}|\tilde{w}|^{2}\,dt + \sum_{i,j=1}^{n}\E\int_{Q_{ij}^\ast}s^{-1}e^{2s\varphi}|D_{ij}^2\tilde{w}|^2\, dt + \E\int_{Q}s\lambda_{1}^2\varphi\,e^{2s\varphi}|\tilde{g}|^2\,dt \\
&\leq \mathcal{C}_{\tau_0,\varepsilon_0}\left(\sum_{i=1}^{n}\E\int_{Q}\,e^{2s\varphi}|a_{1i}A_iD_i(\tilde{w})|^2\,dt + \E\int_{Q}\,e^{2s\varphi}|a_2\tilde{w}|^2\,dt + \sum_{i=1}^{n}\E\int_{Q_i^{\ast}}\,e^{2s\varphi}|A_i\tilde{g}|^2\,dt \right. \\
&\left. + \sum_{i=1}^{n}\E\int_0^T\int_{\Gamma}st_{r}^{i}(e^{2s\varphi}|D_i\tilde{w}|^2)\,dt + \left.\E\int_{\mathcal{M}}s^2\,e^{2s\varphi}|\tilde{w}|^2\right|_{t=T}\right).
\end{split}
\end{equation}

Our next objective is to estimate the second term on the right-hand side of the above equation, which we consider as $I$. 

By applying Hölder's inequality with $\frac{1}{r}+\frac{1}{q}=1$ to the second term on the right-hand side, we obtain
\[
I\leq \E\int_{0}^T\left(\int_{\mathcal{M}}|a_{2}|^{2r}\,\right)^{1/r}\left(\int_{\mathcal{M}}|e^{s\varphi}\tilde{w}|^{2q}\right)^{1/q}dt.
\]
Applying { \cite[Theorem 1.5]{LPP:2025} } to the second term from the right-hand side above  with $p=2$ and $p^{\ast}=2q$, we obtain the following.
\[
I\leq \mathcal{C}\E\int_{0}^T\left(\int_{\mathcal{M}}|a_{2}|^{2r}\,\right)^{2/2r}\left(\int_{\mathcal{M}}|e^{s\varphi}\tilde{w}|^{2}+\sum_{i=1}^{n}\int_{\mathcal{M}_i^{\ast}}|D_i(e^{s\varphi}\tilde{w})|^2\right)dt,
\]
where $\mathcal{C}$ is a constant independent of $h$ and if $n>2$ thus $2r=n$ or if $n=2$ thus $ 2r\in (2,+\infty)$. Taking into account \eqref{eq:difference:product} and the triangle inequality, from the above inequality, we deduce
\begin{align*}
I\leq \mathcal{C}\|a_{2}\|^2_{L^{\infty}_{\mathcal{F}}(0,T;L_{h}^{2r}(\mathcal{M}))}\E\int_0^T\left(\int_{\mathcal{M}}|e^{s\varphi}\tilde{w}|^{2}+\sum_{i=1}^{n}\int_{\mathcal{M}_{i}^\ast}|D_ie^{s\varphi}|^2|A_i\tilde{w}|^2+|A_ie^{s\varphi}|^2|D_i\tilde{w}|^2\right)dt.
\end{align*}
Notice that thanks to \eqref{eq:average:product}, for $i=1,\ldots,n$ and $u\in C(\mathcal{M})$ we obtain the following inequality.
\begin{equation}\label{eq:inequalityAverange}
    A_i(|u|^2)\geq |A_iu|^2.
\end{equation}
Thus, from the above inequality, \eqref{eq:int:ave}, and using the boundary condition $\tilde{w}=0$ on $\partial \mathcal{M}$, it follows that
\[
\int_{\mathcal{M}_{i}^\ast}|D_ie^{s\varphi}|^2|A_i\tilde{w}|^2\leq \int_{\mathcal{M}_i^{\ast}}|D_ie^{s\varphi}|^2A_i(|\tilde{w}|^2)= \int_{\mathcal{M}}A_i(|D_ie^{s\varphi}|^2)|\tilde{w}|^2
\]
Therefore, we can assert that 
\begin{align*}
I\leq \mathcal{C}\|a_{2}\|^2_{L^{\infty}_{\mathcal{F}}(0,T;L_{h}^{2r}(\mathcal{M}))}\E\int_0^T\left(\int_{\mathcal{M}}\left(e^{2s\varphi}+\sum_{i=1}^{n}A_i(|D_ie^{s\varphi}|^2)\right)\right.&|\tilde{w}|^{2}\\
+\sum_{i=1}^{n}&\left.\int_{\mathcal{M}_i^{\ast}}|A_ie^{s\varphi}|^2|D_i\tilde{w}|^2\right)dt.
\end{align*}

Finally, taking $n^*\ge2r$ and using the bounds $\displaystyle A_i(|D_i(e^{s\varphi})|^2)\leq \lambda^2s^2\varphi^2(\partial_i\psi)^2e^{2s\varphi}+e^{2s\varphi}\mathcal{O}_{\lambda}((sh)^2)$ and $|A_i(e^{s\varphi})|^2\leq e^{2s\varphi}+e^{2s\varphi}\mathcal{O}_{\lambda}((sh)^2)$, we conclude that
\begin{equation}\label{eq:conditionn*}
    \begin{split}
I\leq \mathcal{C}\|a_{2}\|^2_{L^{\infty}_{\mathcal{F}}(0,T;L_{h}^{n^{\ast}}(\mathcal{M}))}\E\int_0^T&\left(\int_{\mathcal{M}}\left(1+\lambda^2s^2\varphi^2|\nabla\psi|^2\right)e^{2s\varphi}|\tilde{w}|^{2}+\sum_{i=1}^{n}\int_{\mathcal{M}_i^{\ast}}e^{2s\varphi}|D_i\tilde{w}|^2\right)\,dt\\
&+\mathcal{O}_{\lambda}((sh)^2)\E\int_{0}^T\left(\int_{\mathcal{M}}e^{2s\varphi}|\tilde{w}|^2+\sum_{i=1}^{n}\int_{\mathcal{M}_i^{\ast}}e^{2s\varphi}|D_i\tilde{w}|^2\right)\,dt,
    \end{split}
\end{equation}
where we check at once that $n^\ast$ satisfies condition \eqref{eq:conditionn*}.

On the other hand, let us focus on the term $\displaystyle \E\int_{Q_i^{\ast}}|A_i\tilde{g}|^2\,dt$
in \eqref{eq:inequalityEC}. Applying \eqref{eq:inequalityAverange} and performing integration by parts with respect to the averaging operator \eqref{eq:int:ave}, we obtain the following.
\begin{equation}\label{eq1:Ag}
    \E\int_{Q_i^{\ast}}\,e^{2s\varphi}|A_i\tilde{g}|^2\,dt \leq \E\int_{Q_i^{\ast}}e^{2s\varphi}A_i(|\tilde{g}|^2)\,dt = \E\int_{Q}A_i(e^{2s\varphi})|\tilde{g}|^2\,dt + \frac{h}{2}\E\int_{\partial_i Q}t_r^{i}(e^{2s\varphi})|\tilde{g}|^2\,dt.
\end{equation}
To deal with the last term in the above inequality, consider $x_i=e_{i}\cdot x$, we can assert that
\[
\begin{aligned}
\int_{\mathcal{M}}(2x_i-1)D_i(e^{2s\varphi})|\tilde{g}|^2=-\int_{\mathcal{M}_i^{*}}2e^{2s\varphi}A_i(|\tilde{g}|^2)-\int_{\mathcal{M}_i^*}A_{i}(2x_i-1)e^{2s\varphi}D_i(|\tilde{g}|^2)\\+\int_{\partial_i\mathcal{M}}(2x_{i}-1)\nu_{i}t_{r}^{i}(e^{2s\varphi})|\tilde{g}|^2.
\end{aligned}
\]
Now, using that $D_i(|\tilde{g}|^2)=2A_i(\tilde{g})D_i(\tilde{g})$ on second term in the right-hand side above and Young inequality on the this result, we see that
\[
\begin{aligned}
 \int_{\partial_i\mathcal{M}}(2x_{i}-1)\nu_{i}&t_{r}^{i}(e^{2s\varphi})|\tilde{g}|^2\leq \int_{\mathcal{M}}(2x_i-1)D_i(e^{2s\varphi})|\tilde{g}|^2+\int_{\mathcal{M}_i^{\ast}}2e^{2s\varphi}A_i(|\tilde{g}|^2)\\
 &+\int_{\mathcal{M}_i^{*}}A_i(2x_i-1)e^{2s\varphi}|A_i\tilde{g}|^2+\int_{\mathcal{M}_i}A_i(2x_i-1)e^{2s\varphi}|D_i\tilde{g}|^2.
\end{aligned}
\]
Applying again the condition \eqref{conditionong} and the fact $|A_i\tilde{g}|^2\leq A_i(|\tilde{g}|^2)$, we conclude that
\[
\begin{aligned}
    \int_{\partial_i\mathcal{M}}(2x_{i}-1)\nu_{i}t_{r}^{i}(e^{2s\varphi})|\tilde{g}|^2\leq &\int_{\mathcal{M}}(2x_i-1)D_i(e^{2s\varphi})|\tilde{g}|^2\\
    &+2\int_{\mathcal{M}_i^{\ast}}(1+A_i(2x_i-1))e^{2s\varphi}A_i(|\tilde{g}|^2).
\end{aligned}
\]
Moreover, using that $|1+A_i(2x_i-1)|\leq 2$ for all $x_i$ and integration by parts with respect to average operator, it follows that
\[
\int_{\partial_i\mathcal{M}}(2x_{i}-1)\nu_{i}t_{r}^{i}(e^{2s\varphi})|\tilde{g}|^2\leq \int_{\mathcal{M}}(2x_i-1)D_i(e^{2s\varphi})|\tilde{g}|^2+4\int_{\mathcal{M}}A_i(e^{2s\varphi})|\tilde{g}|^2+2h\int_{\partial_i\mathcal{M}}t_{r}^{i}(e^{2s\varphi})|\tilde{g}|^2
\]
Taking $h_1\leq \frac{1}{4}$ and noting that $(2x_i-1)\nu_i=1$ for $x_i\in \partial_i\mathcal{M}$, we have
\[
\frac{1}{2}\int_{\partial_i\mathcal{M}}t_{r}^{i}(e^{2s\varphi})|\tilde{g}|^2\leq\int_{\mathcal{M}}(2x_i-1)D_i(e^{2s\varphi})|\tilde{g}|^2+4\int_{\mathcal{M}}A_i(e^{2s\varphi})|\tilde{g}|^2
\]
for all $0<h<h_1$.

Combining the above equation and \eqref{eq1:Ag} with $e^{-2s\varphi}D_i(e^{2s\varphi})=s\lambda\varphi\partial_i\psi+\mathcal{O}_{\lambda}((sh)^2)$ and $e^{-2s\varphi}A_i(e^{2s\varphi}) = 1 + \mathcal{O}_{\lambda}((sh)^2)$, due to Proposition 2.4 in \cite{AA:perez:2024}, respectively, we obtain
\begin{equation*}
\E\int_{Q_i^{\ast}}\,e^{2s\varphi}|A_i\tilde{g}|^2\,dt \leq \E\int_{Q}\left(1+\mathcal{O}_{\lambda}(sh)+\mathcal{O}_{\lambda}(h) + \mathcal{O}_{\lambda}((sh)^2)\right)e^{2s\varphi}|\tilde{g}|^2\,dt.
\end{equation*}
for all $0<h<\min\{h_0,h_{1}\}$.

Finally, recalling that $s(t):=\tau\theta(t)$, using the above inequality in the third term on the right-hand side of \eqref{eq:inequalityEC}, the estimate of $I$ and letting $\tau \geq C$, we can absorb the initial three terms on the right-hand side of the same inequality \eqref{eq:inequalityEC}. Therefore, we can achieve the following inequality,
{ 
\begin{equation*}
\begin{split}
&J(\tilde{w}) + \E\int_{Q}s^{3}\lambda^{4}_{1}\varphi^{3}\,e^{2s\varphi}|\tilde{w}|^{2}\,dt +{ \sum_{i,j=1}^{n}}\E\int_{Q_{ij}^{\ast}}s^{-1}e^{2s\varphi}|D_{ij}^2\tilde{w}|^2\, dt + \E\int_{Q}s\lambda_{1}^2\varphi\,e^{2s\varphi}|\tilde{g}|^2\,dt \\
&\leq \mathcal{C}\left(\int_{0}^T\int_{G_0\cap\mathcal{M}}s^3\varphi e^{2s\varphi}|\tilde{w}|\,dt + \left.\E\int_{\mathcal{M}}s^2\,e^{2s\varphi}|\tilde{w}|^2\right|_{t=T}\right).
\end{split}
\end{equation*}
}
After dropping the positive terms, from the left-hand side of above inequality and recalling the definition of the operator $\Lambda_1$, we deduce \eqref{eq:firstresult} and the proof of the theorem is completed.

\section{Semi-discrete Cauchy problem.}\label{sec:cauchyproblem}

In this section, we establish the conditional stability result for the semi-discrete Cauchy problem, as stated in Theorem~\ref{Teo:SecondResult}. The section is structured as follows: { First, we introduce a new weight function in the semi-discrete context to derive a new semi-discrete Carleman estimate with boundary data, using a similar procedure for obtaining \eqref{theo:Carleman_datainterior}. Second, we need an auxiliary estimate for the solution of a semi-discrete stochastic parabolic equation with non-homogeneous Dirichlet boundary conditions. This estimate plays a crucial role in the proof a new Carleman estimate \eqref{eq:firstCarlemanV2} to this setting. Second, we present a modified Carleman estimate adapted to the semi-discrete stochastic parabolic operator with nonhomogeneous boundary data, and provide its proof. Finally, we prove Theorem~\ref{Teo:SecondResult} by combining the Carleman estimate obtained in the second step with geometric arguments inspired by \cite{Yamamoto_2009}.

We begin by deriving the new weight function required for the semi-discrete Carleman estimate with boundary data. Following the strategy used in \cite{BHLR:2010b} to obtain Assumption~\ref{Assumption_data_inteior}, and adapting the continuous construction of \cite[Lemma~1.2]{Emanuilov:1995} to the semi-discrete framework, we obtain the existence of such a function as follows:
\begin{assumption}\label{WFdataboundary}
Let $\tilde{\Gamma}\subsetneq\Gamma$ be an open set. Let $\tilde{G}$ be a smooth open and connected neighborhood of $\overline{G}$ in $\mathbb{R}^n$. The function $x\mapsto d_0(x)$ in $C^{k}(\overline{\tilde{G}},\mathbb{R})$, $k$ is sufficiently large and satisfies
\begin{equation*}
    \begin{array}{ccc}
         d_0(x)>0\quad \text{in}\,\, \tilde{G},&\quad |\nabla d_0(x)|>c\quad \text{in}\,\,x\in  \overline{\tilde{G}}, & \sum_{i=1}^{n} \partial_id_0(x) \nu_{i}(x)\leq 0, \quad x\in  V_{\partial_iG\setminus\Gamma},    
    \end{array}
\end{equation*}
where $V_{\partial_iG\setminus \Gamma}$ is a sufficiently small neighborhood of $\partial_iG\setminus \Gamma$ in $\tilde{G}\setminus \tilde{\Gamma}$, in which the outward normal unit $\nu_i$ to $G$ extends from $\partial_i G$.
\end{assumption}
\begin{example} Let $ x^{*} \notin \overline{G} $. We define the set $\mathcal{M}$ and  
\[
\Gamma := \left\{ x \in \partial \mathcal{M} : x \in \partial_i \mathcal{M} \ \text{and}\ (x - x^{*}) \cdot e_i\,\nu_i > 0 \right\}.
\]
For these sets, a function that satisfies the assumption \ref{WFdataboundary} is $d_0(x) = |x - x^{*}|^2.$

In this construction, the size of the observation set $\Gamma$ depends on the choice of the reference point $x^{*}$. Indeed, in dimension $ 2 $, let us consider a point $ x^{*} = (a,b) $ with $ a < 0$  and $ 0 < b < 1 $. In this case, if we consider $L:=(0,1)\cap h\mathbb{Z}$, then $\Gamma= \{(x,0)|\,x\in L\} \cup \{(1,x)|\, x\in L \} \cup \{(x,1)|\,x\in L\}.$ If instead $b>1$, then $\Gamma= \{(x,0)|\,x\in L\} \cup \{(1,x)|\,x\in L\},$ while for  $b<0$ we have $\Gamma= \{(1,x)|\,x\in L\} \cup \{(x,1)|\,x\in L\}.
$ Therefore, the set $\Gamma$ contains more faces of the square when  $0< b < 1$  than when $b > 1$ or $b < 0$. In general, if $x^{*}$ is chosen so that there exists some $x \in \partial \mathcal{M}$ for which the vector $x - x^{*}$ is orthogonal to a face of $\partial \mathcal{M}$, then the corresponding portion of $\Gamma$ becomes larger. In the opposite case, the set $\Gamma$ becomes smaller. Thus, a suitable choice of $x^{*}$ can be used to minimize the required observation region.
\end{example}

The following result may be proved in much the same way as Theorem \ref{theo:Carleman_datainterior}.

\begin{theorem}\label{theo:CarlemanDataboundary}
Let $d_0$ satisfy assumptions \ref{WFdataboundary} and $\varphi=e^{\lambda d_0}$. For $\lambda_1\geq 1$ sufficiently large, there exist $C$, $\tau_{0}\geq 1$, $h_{0}>0$, $\varepsilon_0 >0$, depending on $G$, $\beta$, $c_{0}$, $T$, and $\lambda$, such that 
\begin{equation}\label{eq:firstCarleman}
\begin{split}
&J(w)+\E\int_{Q}s^{3}\lambda^{4}_{1}\varphi^{3}\,e^{2s\varphi}|w|^{2}\,dt+\sum_{i,j=1}^{n}\E\int_{Q_{ij}^{\ast}}s^{-1}e^{2s\varphi}|D_{ij}^2w|^2\, dt+\E\int_{Q}s\lambda_{1}^2\varphi\,e^{2s\varphi}|g|^2dt\\
&\leq\,\mathcal{C}_{\tau_{0},\varepsilon_0}\left(\E\int_{Q}\,e^{2s\varphi}|f|^2\,dt+\sum_{i=1}^{n}\E\int_{Q_i^{\ast}}\,e^{2s\varphi}|D_{i}g|^2dt+\sum_{i=1}^{n}\E\int_0^T\int_{\Gamma}st_{r}^{i}(e^{2s\varphi}|D_iw|^2)\,dt\right.\\
&\left.+\left.\E\int_{\mathcal{M}}s^2\,e^{2s\varphi}|w|^2\right|_{t=T}\right).
\end{split}
\end{equation}
for all $\tau\geq \tau_0$, $0<h\leq h_0$, $s(t)h\leq \epsilon_0$, where $f\in L^{2}_{\mathcal{F}}(0,T;L^2(\mathcal{M}))$, $g\in L^{2}_{\mathcal{F}}(0,T;H^1(\mathcal{M}))$, and $w$ satisfy $\displaystyle dw-\sum_{i=1}^{n} D_i(\gamma_i D_{i}w)\,dt=fdt+gdB(t)$ with $w=0$ on $\partial \mathcal{M}$ and $w|_{t=0}=0$ in $\mathcal{M}$.
\end{theorem}
\begin{proof}
   We only outline the main ideas of the proof. If we replace $ \psi(x) $ by $ d_0(x)$ given in assumption \eqref{WFdataboundary} in Appendix~\ref{Proof:CarlemanEstimateHomegeneuos}, all computations remain unchanged, except for the estimate in \eqref{eq:finalestimateinz2}, due to the different properties of the new weight function. Thus, in \eqref{eq:finalestimateinz2}, the terms supported in $ G_0 \cap \mathcal{M} $ are replaced by boundary trace terms of the derivatives on a portion of the boundary. Indeed, for all $\tau\geq \tau_{1}>1$, $0<h<h_0$, and $s(t)h\le\varepsilon_1$, we obtain
\begin{equation}\label{eq1:carlemandatainterior}
\begin{aligned}
    \E\int_{Q}& s^{3}\lambda^{4}_{1}\varphi^{3}|z|^{2}\,dt+\E\int_{Q}(s^2\lambda_{1}^2\varphi^2+s\lambda_{1}^2\varphi)|dz|^2+\E\int_{Q}|Cz|^2\,dt\\
&+\sum_{i=1}^{n}\left(\E\int_{Q_{i}^{\ast}}s\lambda^{2}_{1}\varphi\,|D_{i}z|^{2}\,dt+\E\int_{Q} s\lambda_{1}^2\varphi|A_iD_iz|^2\,dt-\E\int_{Q_i^{\ast}}|D_i(dz)|^2\right)\\
&\leq\mathcal{C}_{\tau_{0},\varepsilon_0}\left(\E\int_{Q}|rf|^2\,dt+\sum_{i=1}^{n}\E\int_{0}^{T}\int_{\Gamma}st_{r}^{i}(|D_iz|^2)\,dt+\left.\E\int_{\mathcal{M}}s^2|z|^2\right|_{t=T}\right),
\end{aligned}
\end{equation}
where $z=e^{-s\varphi}w$, and $w$ satisfies $\displaystyle
dw-\sum_{i=1}^{n}D_i(\gamma_iD_iw)\,dt=f\,dt+g\,dB(t).$

We now return to the original variable $w$. The argument used in Section~\ref{sec:returninbackvarible} applies with minor modifications. The only change is in the right-hand side of \eqref{eq1:carlemandatainterior}, which must be adapted accordingly; the details are given in Appendix~\ref{appendix:additionalreturn}. This yields the Carleman estimate with boundary data \eqref{eq:firstCarleman}.
\end{proof}

From \eqref{eq:firstCarleman}, it may be conclude a result of Lipschitz stability for the semi-discrete inverse random source problem with boundary data that generalizes the result found in \cite{WWW:2024} to arbitrary dimensions as
\begin{theorem}
Let $a_{1i}\in L^\infty_{\mathcal{F}}(0,T;L_h^\infty(\mathcal{M}))$ for $i=1,\ldots,n$, and $a_2\in L_{\mathcal{F}}^{\infty}(0,T;L_h^{n^{\ast}}(\mathcal{M}))$ where $n^{\ast}$ satisfies \eqref{eq:conditionn*}. Assume that $g_k \in L^2_{\mathcal{F}}(0,T;H_h^1(\mathcal{M}))$ and satisfies the condition \eqref{conditionong} for $k=1,2$. Then, there exists a constant $\mathcal{C} > 0$, independent of $h$, such that
\begin{equation}
\|g_1 - g_2\|_{L^2_{\mathcal{F}}(0,T;L_h^2(\mathcal{M}))} \leq \mathcal{C} \|\Lambda_3(g_1) - \Lambda_3(g_2)\|_{(L^2_{\mathcal{F}}(0,T;L^2_{h}(\Gamma)))^n\times L^2_{\mathcal{F}_T}(\Omega,L^2_{h}(\mathcal{M}))}.
\end{equation}
where $\Lambda_{3}$ is defined as
\[
\Lambda_{3}(g)=(\left.\partial_{\nu}w\right|_{(0,T)\times\Gamma},\left.w\right|_{t=T})
\]
for $w$ solution of \eqref{systemrandomsource}.
\end{theorem}

We continue in this fashion obtaining the following result about the estimate for the solution of a semi-discrete homogeneous stochastic parabolic equation with non-homogeneous Dirichlet boundary conditions.
}
\begin{lemma}\label{lemmaoftheu}
Let $\xi \in H_{\mathcal{F}}^1(0,T; H_h^{1/2}(\partial \mathcal{M}) \cap H_h^1(\partial \mathcal{M}))$ with $\left.\xi\right|_{t=0} = 0$. Then, for a.e. $\omega \in \Omega$, the solution $u \in L^2(0,T; H_h^2(\mathcal{M})) \cap C(0,T; L_h^2(\mathcal{M}))$ of the system:
\begin{equation}\label{sys:homogeneuous}
    \begin{cases}
    du - \sum_{i=1}^{n} D_i(\gamma_i D_i u) \, dt = 0, & \text{in } Q, \\
    u = \xi,\quad\text{on } \partial Q,\quad \left.u\right|_{t=0} = 0,\quad  \text{in } \mathcal{M},
    \end{cases}
\end{equation}
satisfies the following stability estimate for any constant $\mathcal{C} > 0$:
\begin{equation}\label{ine:stability:H2}
    \|u\|_{L^2_{\mathcal{F}}(0,T; H_h^2(\mathcal{M}))} + \|u\|_{L^2_{\mathcal{F}}(\Omega; C([0,T]; L_h^2(\mathcal{M})))} \leq \mathcal{C} \|\xi\|_{H^1_{\mathcal{F}}(0,T; H_h^{1/2}(\partial \mathcal{M}) \cap H_h^1(\partial \mathcal{M}))}.
\end{equation}
Moreover, the following boundary estimate holds for $t_r^i(D_i u)$ on $\partial\mathcal{M}$:
\begin{equation}\label{ine:stability:boundary}
    \sum_{i=1}^{n} \mathbb{E} \int_{\partial Q} t_r^i(|D_i u|^2) \, dt \leq \mathcal{C} \|\xi\|_{L^2_{\mathcal{F}}(0,T; H_h^{1/2}(\partial \mathcal{M}) \cap H_h^1(\partial \mathcal{M}))}.
\end{equation}
\end{lemma}
\begin{proof}
Noting that for a.e. $\omega \in \Omega$ and $t \in [0,T]$, $\xi \in H_h^{1/2}(\partial \mathcal{M})$, there exists $v \in H_h^1(\mathcal{M})$ such that for a.e. $\omega \in \Omega$ and $t \in [0,T]$, $v = \xi$ on $\partial \mathcal{M}$ and
\begin{equation}\label{eq:tracexi}
\|v\|_{H_h^1(\mathcal{M})} =  \|\xi\|_{H_h^{1/2}(\partial \mathcal{M})}.
\end{equation}
Let $\tilde{u} = u - v$. Since $v|_{t=0}=0$ by \eqref{eq:tracexi} and $\xi|_{t=0}=0$, it follows that $\tilde{u}|_{t=0} = u|_{t=0} - v|_{t=0} = 0$. Moreover, on $\partial\mathcal{M}\times(0,T)$ we have $\tilde{u} = \xi - \xi = 0$. Therefore, $\tilde{u}$ satisfies
\begin{equation}\label{proof:sys:homogeneuous}
    \begin{cases}
    d\tilde{u} - \sum_{i=1}^{n} D_i(\gamma_i D_i \tilde{u}) \, dt = \left(\partial_t v - \sum_{i=1}^{n} D_i(\gamma_i D_i v)\right) \, dt, & \text{in } \mathcal{M} \times (0,T), \\
    \tilde{u} = 0, \quad\quad\text{on } \partial \mathcal{M} \times (0,T),\quad\quad
    \left.\tilde{u}\right|_{t=0} = 0,\quad \text{in } \mathcal{M}.
    \end{cases}
\end{equation}
The task is now to find an energy estimate for $u$. Let us apply the It\^{o}'s formula to $\displaystyle \frac{1}{2}\tilde{u}^2$ yields $\displaystyle \frac{1}{2}d(\tilde{u}^2)=\tilde{u}d\tilde{u}+\frac{1}{2}|d\tilde{u}|^2$. As $\tilde{u}$ satisfies \eqref{proof:sys:homogeneuous}, we have 
\[
\frac{1}{2}d(\tilde{u}^2)=\sum_{i=1}^{n}\tilde{u}D_i(\gamma_iD_i\tilde{u})\,dt+\tilde{u}\partial_{t}v\,dt-\sum_{i=1}^{n}\tilde{u}D_i(\gamma_iD_iv)\,dt
\]
By integration over $\mathcal{M}$ and integration by parts to the first and third terms, we get
\[
\frac{1}{2}\int_{\mathcal{M}}d(\tilde{u}^{2})=-\sum_{i=1}^{n}\int_{\mathcal{M}_i^{\ast}}\gamma_i|D_i(\tilde{u})|^2\,dt+\int_{\mathcal{M}}\tilde{u}\partial_{t}v\,dt+\sum_{i=1}^{n}\int_{\mathcal{M}_{i}^{*}}\gamma_iD_i(\tilde{u})\;D_i(v)\,dt,
\]
which is due to the fact that $\tilde{u}=0$ on $\partial\mathcal{M}$. From Young's inequality on the second and third terms in the above equation, it follows that
\[
\frac{1}{2}\int_{\mathcal{M}}d(\tilde{u}^2)\leq-\frac{1}{2}\sum_{i=1}^{n}\int_{\mathcal{M}_i^{*}}\gamma_i|D_i\tilde{u}|^2\,dt+\frac{1}{2}\int_{\mathcal{M}}|\tilde{u}|^2\,dt+\frac{1}{2}\int_{\mathcal{M}}|\partial_{t}v|^2\,dt+\frac{1}{2}\sum_{i=1}^{n}\int_{\mathcal{M}_i^{*}}\gamma_i|D_iv|^2\,dt
\]
Now, integrating on $[0,t]$ for $t\in [0,T]$, using \eqref{eq:tracexi} of the above inequality and from condition that $\left.\tilde{u}\right|_{t=0}=0$, we can assert that
\begin{equation}\label{eqE}
    \frac{1}{2} \, \int_{\mathcal{M}}|\tilde{u}|^2+ \frac{1}{2}\sum_{i=1}^{n}\int_{Q_i^{*}} \gamma_i|D_i\tilde{u}|^2\, dt \leq \mathcal{C}\int_{0}^{T}\left(\|\xi_t\|^2_{ H_h^{1/2}(\partial \mathcal{M})}+\|\xi\|^2_{H_h^{1/2}(\partial \mathcal{M})}\right)\, dt.
\end{equation}
Recalling $u = \tilde{u} + v$ and using the above inequality, we deduce that
\[
\begin{aligned}
\frac{1}{2}\int_{\mathcal{M}}|u|^2+\frac{1}{2}\sum_{i=1}^{n}\int_{Q_i^{*}}\gamma_i|D_iu|^2\,dt\leq \mathcal{C}\int_{0}^{T}\left(\|\xi_t\|^2_{ H_h^{1/2}(\partial \mathcal{M})}+\|\xi\|^2_{H_h^{1/2}(\partial \mathcal{M})}\right)\, dt+\frac{1}{2}\int_{\mathcal{M}}|v|^2.
\end{aligned}
\]
To estimate the term $\displaystyle \int_{\mathcal{M}}|v|^2$, we note that applying Young's inequality and \eqref{eq:tracexi} enables us to write
\[
\frac{1}{2}\int_{\mathcal{M}}\partial_{t}|v|^2\leq \frac{1}{2}\int_{\mathcal{M}}v\,\partial_{t}v\leq \frac{1}{2}\int_{\mathcal{M}}|v|^2+|v_{t}|^2\leq \frac{1}{2}\left(\|\xi\|^2_{H^{1/2}_{h}(\partial \mathcal{M})}+\|\xi_{t}\|_{H^{1/2}_{h}(\partial\mathcal{M})}\right)  
\]
Thus, 
\[
\int_{\mathcal{M}}|v|^2\leq \int_0^{T}\|\xi_{t}\|^2_{H_{h}^{1/2}(\partial\mathcal{M})}+\|\xi\|^2_{H_{h}^{1/2}(\partial \mathcal{M})}\,dt
\]
where we used the integration on $[0,T]$ and that $\left.v\right|_{t=0}=0$.
Combining these inequalities, we conclude that
\begin{equation}\label{eq:estimateenergy}
  \frac{1}{2}\int_{\mathcal{M}}|u|^2+\frac{1}{2}\sum_{i=1}^{n}\int_{Q_i^{*}}\gamma_i|D_iu|^2\,dt\leq \mathcal{C}\int_{0}^{T}\left(\|\xi_t\|^2_{ H_h^{1/2}(\partial \mathcal{M})}+\|\xi\|^2_{H_h^{1/2}(\partial \mathcal{M})}\right)\, dt.
\end{equation}
Therefore, taking the supremum over $[0,T]$ in the above inequality only considered the first terms, we conclude that
\begin{equation*}
\|u\|_{C(0,T; L_h^2(\mathcal{M}))} \leq \mathcal{C}\int_{0}^{T} \|\xi_{t}\|^2_{H^{1/2}_{h}(\partial \mathcal{M})}+\|\xi\|^2_{H_{h}^{1/2}(\partial\mathcal{M})} dt.
\end{equation*}
Furthermore, using \eqref{eq:estimateenergy}, the above inequality and integrating over $(0,T)$, we have
\begin{equation*}
\|u\|_{L_h^2(0,T; H^1_{h}(\mathcal{M}))} \leq\mathcal{C}\int_{0}^{T} \|\xi_{t}\|^2_{H^{1/2}_{h}(\partial \mathcal{M})}+\|\xi\|^2_{H_{h}^{1/2}(\partial\mathcal{M})} dt.
\end{equation*}
Similar arguments apply to the case $D_iu$ for each $i=1,\ldots,n$, yields
\[
\|u\|_{L_h^2(0,T; H_{h}^2(\mathcal{M}))} \leq\mathcal{C}\int_{0}^{T} \|\xi_{t}\|^2_{H^{1/2}_{h}(\partial \mathcal{M})\cap H^{1}(\partial\mathcal{M})}+\|\xi\|^2_{H_{h}^{1/2}(\partial\mathcal{M})\cap H^1_{h}(\partial\mathcal{M})} dt.
\]
Combining these inequalities and taking the expectation, we derive the first estimate \eqref{ine:stability:H2}.

Our next goal is to determine the estimate for $\displaystyle \sum_{i=1}^{n}\E\int_{\partial Q}t_{r}^{i}(|D_iu|^2)\,dt$. In order to obtain this inequality, it is convenient to consider the following 
\begin{equation}\label{eq:identityD}
2x_{j}D_{j}^2(u)A_{j}D_ju=x_{j}D_{j}(|D_ju|^2)
\end{equation}
where $x_{j}=x\cdot e_{j}$.

Integrating the results over $\mathcal{M}\times(0,T)$ and applying an integration by parts it follows
\begin{align*}
    \frac{1}{2}\int_{\mathcal{M}}x_{j}D_{j}(|D_{j}u|^{2})=&-\frac{1}{2}\int_{\mathcal{M}^{\ast}_{j}}D_{j}(x_{j})|D_{j}u|^{2}+\frac{1}{2}\int_{\partial_{j}\mathcal{M}}x_{j}t_{r}^{j}(|D_{j}u|^{2})\nu_{j}.
\end{align*}
Thus, denoting $\Gamma_{j}:=\{ x\in\partial_{j}\mathcal{M}\mid x_{j}\nu_{j}= 1\}\subset\partial_j\mathcal{M}$ it follows that the above equals
\begin{align}\label{eq:boundary:1}
    \frac{1}{2}\int_{\mathcal{M}}x_{j}D_{j}(|D_{j}u|^{2})=-\frac{1}{2}\int_{\mathcal{M}_{j}^{\ast}}|D_{j}u|^2+\frac{1}{2}\int_{\Gamma_{j}}t_{r}^{j}(|D_{j}u|^{2}).
\end{align}
Hence, integrating over $\mathcal{M}\times (0,T)$ in \eqref{eq:identityD}, and then  using equation \eqref{eq:boundary:1} we obtain
\begin{equation}
\begin{aligned}
\int_{Q}\,x_{j}D_j^2(u)A_{j}D_{j}(u)\,\,dt=-\frac{1}{2}\int_{\mathcal{M}_{j}^{\ast}}|D_{j}u|^2+&\frac{1}{2}\int_{\Gamma_{j}\times(0,T)}t_{r}^{j}(|D_{j}u|^{2})\,dt
\end{aligned}
\end{equation}
Finally, using the Young inequality and taking expectation we conclude that there exists a constant $\mathcal{C}$ independent of $h$ such that
\begin{equation}\label{ine:boundary:H^{2}}
    \E\int_{\Gamma_{j}\times(0,T)}t_{r}^{j}(|D_{j}u|^{2})dt\leq \mathcal{C}\left\| u\right\|^{2}_{L^{2}_{\mathcal{F}}(0,T;H_h^{2}(\mathcal{M}))}.
\end{equation}
The inequality \eqref{ine:boundary:H^{2}} stands for a bound for the elements of the boundary $\Gamma_{j}$, in the direction $x_{j}$, such that its normal function is equal to one. Similarly, it is possible to obtain a bound for the normal derivative on the boundary nodes where the normal function defined in \eqref{normal:function} is minus one, that is, to adopt the previous methodology by using the term $(1-x_{j})A_{j}D_{j}$ instead of the used to obtain the equation \eqref{eq:identityD}. Moreover, we emphasize that we can obtain the inequality \eqref{ine:boundary:H^{2}} in any direction. In conclusion, the solution of the system \eqref{sys:homogeneuous} satisfies
\begin{equation}\label{ine:boundary:normal:derivative}
   \sum_{i=1}^{n}\E\int_{\partial Q}t_{r}^{i}(|D_{i}u|^{2})\,dt\leq 
   \mathcal{C} \left\| u\right\|^{2}_{L_{\mathcal{F}}^{2}(0,T;H_h^{2}(\mathcal{M}))}. 
\end{equation}
Overall, combining \eqref{ine:stability:H2} and \eqref{ine:boundary:normal:derivative} we obtain \eqref{ine:stability:boundary}.
\end{proof}
\begin{remark}
We emphasize that the key argument of the proof of Lemma \ref{lemmaoftheu}  is that we are dealing with square domains.
\end{remark}

\subsection{Semi-discrete Carleman estimate with nonhomogeneous Dirichlet condition.}
Based on this choice of $\psi=d_0$, we now state the corresponding Carleman estimate, which extends \eqref{eq:firstCarleman} to the present setting. 
For completeness, we also provide its proof in the following.

\begin{theorem}\label{theo: CarlemanNonhomgeneous.}
    Let $d_0$ satisfying the Assumption \ref{WFdataboundary} and $\varphi=e^{\lambda d_0}$. For $\lambda\geq 1$ sufficiently large, there exist $\mathcal{C}$, $\tau_{0}\geq 1$, $h_{0}>0$, $\varepsilon_0 >0$, depending on $G$,$\beta$ $c_{0}$, $T$, and $\lambda$, such that 
\begin{equation}\label{eq:secondCarleman}
\begin{split}
&J(w)+\E\int_{Q}s^{3}\lambda^{4}_{1}\varphi^{3}\,e^{2s\varphi}|w|^{2}\,dt+\sum_{i,j=1}^{n}\E\int_{Q_{ij}^\ast}s^{-1}e^{2s\varphi}|D_{ij}^2w|^2\, dt+\E\int_{Q}s\lambda_{1}^2\varphi\,e^{2s\varphi}|g|^2dt\\
&\leq\,\mathcal{C}_{\tau_{0},\varepsilon_0}\left(\E\int_{Q}\,e^{2s\varphi}|f|^2\,dt+\sum_{i=1}^{n}\E\int_{Q_i^{\ast}}\,e^{2s\varphi}|D_{i}g|^2dt+\sum_{i=1}^{n}\E\int_0^T 
\int_{\Gamma}st_{r}^{i}(e^{2s\varphi}|D_iw|^2)\,dt\right.\\
&\left.+\left.\E\int_{\mathcal{M}}s^2\,e^{2s\varphi}|w|^2\right|_{t=T}\right)+\mathcal{C}\tau^3e^{\tau\mathcal{C}}\|\xi\|^2_{H_{\mathcal{F}}^{1}(0,T;H_{h}^{1/2}(\partial\mathcal{M})\cap H_{h}^1(\partial\mathcal{M}))}.
\end{split}
\end{equation}
for all $\tau\geq \tau_0$, $0<h\leq h_0$, $sh\leq \epsilon_0$, where $f\in L^{2}_{\mathcal{F}}(0,T;L_h^2(\mathcal{M}))$, $g\in L^{2}_{\mathcal{F}}(0,T;H^1(\mathcal{M}))$,\\
$\xi\in L^2_{\mathcal{F}}(0,T;H_h^1(\partial\mathcal{M})\cap H_{h}^{1/2}(\partial \mathcal{M)})$ satisfying compatibility condition $\left.\xi\right|_{t=0}=0$ for all $x\in\partial\mathcal{M}$  and $w$ satisfies $dw-{ \sum_{i=1}^{n}} D_i(\gamma_i D_{i}w)\,dt=fdt+gdB(t)$ with $w=\xi$ on $\partial Q$ and $\left.w\right|_{t=0}=0$ in $\mathcal{M}$.
\end{theorem}
\begin{proof}
Let $u$ such that satisfies \eqref{sys:homogeneuous} and taking $z=w-u$, we obtain
\begin{equation*}
    \begin{cases}
     dz-\sum_{i=1}^{n}D_i(\gamma_{i}D_i z)=f\,dt+g\,dB(t)    &\text{in}\; Q , \\
     z=0,\quad\quad\text{on}\,\partial Q,\quad\quad\left.z\right|_{t=0}=0 \quad\quad\text{in}\, \mathcal{M}.\\ 
    \end{cases}
\end{equation*}
Applying Theorem~\ref{theo:CarlemanDataboundary} to $z$, we have for all $\tau\geq \tau_0$, $0<h\leq h_0$, $sh\leq \epsilon_0$,
\begin{equation*}
\begin{split}
&J(z)+\E\int_{Q}s^{3}\lambda^{4}_{1}\varphi^{3}\,e^{2s\varphi}|z|^{2}\,dt+\sum_{i,j=1}^{n}\E\int_{Q_{ij}^{\ast}}s^{-1}e^{2s\varphi}|D_{ij}^2z|^2\, dt+\E\int_{Q}s\lambda_{1}^2\varphi\,e^{2s\varphi}|g|^2dt\\
&\leq\,C_{\tau_{0},\varepsilon_0}\left(\E\int_{Q}\,e^{2s\varphi}|f|^2\,dt+\sum_{i=1}^{n}\E\int_{Q_i^{\ast}}\,e^{2s\varphi}|D_{i}g|^2dt+\sum_{i=1}^{n}\E\int_0^T\int_{\Gamma}st_{r}^{i}(e^{2s\varphi}|D_iz|^2)\,dt\right.\\
&\left.+\left.\E\int_{\mathcal{M}}s^2\,e^{2s\varphi}|z|^2\right|_{t=T}\right).
\end{split}
\end{equation*}
Recalling $w=z+u$ and using the triangular inequality, we obtain
\begin{equation}\label{eq:fiC}
\begin{split}
&J(w)+\E\int_{Q}s^{3}\lambda^{4}_{1}\varphi^{3}\,e^{2s\varphi}|w|^{2}\,dt+\sum_{i,j=1}^{n}\E\int_{Q_{ij}^\ast}s^{-1}e^{2s\varphi}|D_{ij}^2w|^2\, dt+\E\int_{Q}s\lambda_{1}^2\varphi\,e^{2s\varphi}|g|^2dt\\
&\leq\,C_{\tau_{0},\varepsilon_0}\left(\E\int_{Q}\,e^{2s\varphi}|f|^2\,dt+\sum_{i=1}^{n}\E\int_{Q_i^{\ast}}\,e^{2s\varphi}|D_{i}g|^2dt+\sum_{i=1}^{n}\E\int_{\Gamma}st_{r}^{i}(e^{2s\varphi}|D_iw|^2)\,dt\right.\\
&\left.+\left.\E\int_{\mathcal{M}}s^2\,e^{2s\varphi}|w|^2\right|_{t=T}+\sum_{i=1}^{n}\E\int_0^T\int_{\Gamma}st_{r}^{i}(e^{2s\varphi}|D_iu|^2)\,dt+\left.\E\int_{\mathcal{M}}s^2\,e^{2s\varphi}|u|^2\right|_{t=T}\right)\\
&+J(u)+\E\int_{Q}s^{3}\lambda^{4}_{1}\varphi^{3}\,e^{2s\varphi}|u|^{2}\,dt+{ \sum_{i,j=1}^{n}}\E\int_{Q}s^{-1}e^{2s\varphi}|D_{ij}^2u|^2\, dt.
\end{split}
\end{equation}
To estimate the term that depends on $u$ on the right-side hand of the above inequality, we note that
\begin{equation*}
\begin{split}    J(u)+\E\int_{Q}s^{3}\lambda^{4}_{1}\varphi^{3}\,e^{2s\varphi}|u|^{2}&\,dt+\sum_{i,j=1}^{n}\E\int_{Q_{ij}^\ast}s^{-1}e^{2s\varphi}|D_{ij}^2u|^2\, dt+\left.\E\int_{\mathcal{M}}s^2\,e^{2s\varphi}|u|^2\right|_{t=T}\\
&\leq\, \mathcal{C}\tau^3e^{\tau\mathcal{C}}\left(\|u\|^2_{L^2_{\mathcal{F}}(\Omega;L^2((0,T);H_h^2(\mathcal{M}))}+\left.\|u\right|_{t=T}\|^2_{L^2_{\mathcal{F}_T}(\Omega;L_h^2(\mathcal{M})}\right).
\end{split}
\end{equation*}
Now, from Lemma \ref{lemmaoftheu}, we have
\[
\begin{aligned}
    \|u\|&_{L^2_{\mathcal{F}}(\Omega;L^2((0,T);H_h^2(\mathcal{M}))}+\|\left.u\right|_{t=T}\|_{L^2_{\mathcal{F}_T}(\Omega;L_h^2(\mathcal{M})}&\\
    &\leq \mathcal{C}\left(\|u\|_{L^2_{\mathcal{F}}(0,T;H^2_{h}(\mathcal{M}))}+\|u\|_{L^2_{\mathcal{F}}(\Omega;C(0,T;H^2_{h}(\mathcal{M})))}\right)\le \mathcal{C}\|\xi\|_{H_{\mathcal{F}}^{1}(0,T;H_{h}^{1/2}(\partial \mathcal{M})\cap H_{h}^1(\partial \mathcal{M}))}
\end{aligned}
\]
and
\[
\sum_{i=1}^{n}\E\int_{0}^{T}\int_{\Gamma}t_{r}^{i}(|D_iu|^2)\,dt\leq \mathcal{C}\|\xi\|^2_{H_{\mathcal{F}}^{1}(0,T;H_{h}^{1/2}(\partial\mathcal{M})\cap H_{h}^1(\partial\mathcal{M}))}.
\]
Combining these inequalities, we deduce that
\[
\begin{aligned}
    J(u)+\E\int_{Q}s^{3}\lambda^{4}_{1}\varphi^{3}\,e^{2s\varphi}|u|^{2}\,dt+\sum_{i,j=1}^{n}&\E\int_{Q_{ij}^\ast}s^{-1}e^{2s\varphi}|D_{ij}^2u|^2\, dt+\left.\E\int_{\mathcal{M}}s^2\,e^{2s\varphi}|u|^2\right|_{t=T}\\
+\sum_{i=1}^{n}\E\int_{0}^{T}\int_{\Gamma}st_{r}^{i}(e^{s\varphi}|D_iu|^2)\,dt&\leq\, \mathcal{C}\|\xi\|^2_{H_{\mathcal{F}}^{1}(0,T;H_{h}^{1/2}(\partial\mathcal{M})\cap H_{h}^1(\partial\mathcal{M}))}.
\end{aligned}
\]
Finally, combining the above inequality with \eqref{eq:fiC}, the proof is complete.
\end{proof}
Based on the Carleman estimate with non-homogeneous boundary condition \eqref{eq:secondCarleman}, we are now in a position to prove the main stability result of this section, namely Theorem \ref{Teo:SecondResult}.
\subsection{Proof of Theorem \ref{Teo:SecondResult}.}\label{subsec:proofsecondresult}
In order to estimate the solution of \eqref{eq:cauchy problem} in $\mathcal{M}_0\times(\varepsilon,T-\varepsilon)$ by Cauchy data at $\Gamma\times(0,T)$. Thus, according to the geometry of $\mathcal{M}_0$ and $\Gamma$, we need to choose a suitable weight function $\psi$ and $\theta$, respectively. 

For this, we begin by $\psi$.  Choose a bounded domain $G_{1}$ with smooth boundary such that
\[\begin{aligned}
    G\subsetneq G_{1},\quad \Gamma:={\partial \mathcal{M}\cap G_1},\quad \partial \mathcal{M}\setminus \Gamma \subset \partial G_1\cap h\mathbb{Z}^{n}
\end{aligned}
\]
We see that $G_1$ is the union of $G$ with other set $\Tilde{G}$  such that $\partial \Tilde{G}\cap\overline{\mathcal{M}}=\Gamma$ and that $G_1\setminus \overline{G}$ have a non-empty open set. Choosing $\overline{\tilde{G}}_0\subset G_1\setminus\overline{G}$, we apply assumption \ref{Assumption_data_inteior} to obtain $\psi$ satisfying 
\begin{equation*}
    \begin{aligned}
        \psi(x)>0, x\in G_1\quad& \psi(x)=0,\,x\in \partial G_1,&|\nabla \psi(x)|>0,\quad x\in \overline{G_1}\cap \overline{G},
    \end{aligned}
\end{equation*}
where we see at once that $\psi=d_0$ also satisfies assumption \ref{WFdataboundary} with the sets $G$ and $\Gamma$. 

Since $\overline{\mathcal{M}}_0\subset G_1\cap h\mathbb{Z}^{n}$, it follows that for a sufficiently large $\tilde{N}>1$ such that
\begin{equation}
    \mathcal{M}_0\subset\left\{x\in G_1: \psi(x)>\frac{4}{\tilde{N}}\|\psi\|_{C(\overline{G_1})}\right\}\cap\overline{\mathcal{M}}.
\end{equation}
We now turn to $\theta$. We arbitrarily fix $t_0\in [\sqrt{2}\varepsilon,T-\sqrt{2}\varepsilon]$ and we choose $\beta>0$ such that
\begin{equation}
2\beta\varepsilon^2>\|\psi\|_{C(\overline{G_1})}>\beta\varepsilon^2.
\end{equation}

Now, let us consider $\varphi = e^{\lambda \psi}$ and $\theta = e^{-\lambda \beta |t - t_0|^2}$, with fixed large parameter $\lambda>0$ and subject to the previously stated properties of $\psi$, $\beta$, and $t_0$, we defined the following set
\begin{equation}
    G^{(k)}:=\left\{(x,t):x\in \overline{G}, \varphi(x)\theta(t)>\mu_{k}\;\text{such that}\; \mu_{k}=\exp{\left(\lambda\left(\frac{k}{\tilde{N}}\|\psi\|_{C(\overline{G_1})}-\frac{\beta\varepsilon^2}{\tilde{N}}\right)\right)},\right\}.
\end{equation}
with $k=1,2,3,4$.
In the same way as in the proof of Theorem 5.1 in \cite{Yamamoto_2009} it can be verified that
\begin{equation}\label{eq:propertyonewight}
    \mathcal{M}_0\times\left(t_0-\frac{\varepsilon}{\sqrt{\Tilde{N}}},t_0+\frac{\varepsilon}{\sqrt{\Tilde{N}}}\right)\subset G^{(k)}\cap Q\subset \overline{\mathcal{M}}\times (t_0-\sqrt{2}\varepsilon,t_0+\sqrt{2}\varepsilon)
\end{equation}
and
\begin{equation}\label{eq:propertytwoweight}
    \begin{split}
        \partial G^{(1)}\cap\partial Q\subset \Sigma^{(1)}\cup \Sigma^{(2)},&\\
        \Sigma^{(1)}\subset \Gamma\times(0,T),\quad &\quad\Sigma^{(2)}:=\{(x,t):x\in \mathcal{M}, \varphi(x)\theta(t)=\mu_1\}.
    \end{split}
\end{equation}
In order to apply Theorem \ref{theo: CarlemanNonhomgeneous.}, we need a cut-off function because we have no data on $\partial G^{(1)}\setminus (\Gamma\times(0,T))$. Let $\chi\in C^{\infty}(\mathbb{R}^{n+1})$ such that $0\leq \chi \leq 1$ and
\begin{equation}\label{eq:definitionchi}
    \chi(x,t):=\left\{ \begin{array}{cc}
     1,    &\varphi(x)\theta(t)>\mu_3,  \\
     0,    &\varphi(x)\theta(t)<\mu_2. 
    \end{array}\right.
\end{equation}
We set $v=\chi w$, and we have
\begin{equation*}
    dv-\sum_{i=1}^{n}D_i(\gamma_iD_i v)dt=\left(\sum_{i=1}^{n}a_{1i}A_iD_i(v)+a_2v+\mathbf{F}\right)dt+a_3vdB(t),\quad \text{in}\, Q
\end{equation*}
where
\begin{equation*}
    \begin{split}
        \mathbf{F}:= &\chi_{t}w-\sum_{i=1}^{n}D_i(\gamma_iD_i(\chi)A(w))+A_iD_i(\chi)A_i(\gamma_iD_iw)+a_{1i}A_iD_i(\chi)A_i^2(w)\\
        &+\frac{h^2}{4}\left(D_i^2(\chi)D_i(\gamma_iD_i(w))+a_{1i}D_i^2(\chi)A_iD_i(w)+2a_{1i}A_iD_i(\chi)D_i^2(w)\right).  
    \end{split}
\end{equation*}
By \eqref{eq:propertytwoweight} and \eqref{eq:definitionchi}, we see that for all $i{ =1,...,n.}$
\begin{equation}\label{eq:conditionfirstofthev}
    v=D_{i}v=0\quad\quad\text{on}\,\,  \Sigma^{(2)}.
\end{equation}
Since $t_0\in [\sqrt{2}\varepsilon,T-\sqrt{2}\varepsilon]$, we obtain the following result
\begin{equation*}
    \max\left\{\psi(x)-\beta |t_0|^2,\psi(x)-\beta|T-t_0|^2)\right\}\leq \psi(x)-2\beta \varepsilon^2\leq 0,
\end{equation*}
which implies $(\mathcal{M}\times\{t=0\})\cup (\mathcal{M}\times\{t=T\})\nsubseteq \sup{(\chi)}$ and then 
\begin{equation}\label{eq:conditionsecondofthev}
\left.v\right|_{t=0}=\left.v\right|_{t=T}=0.
\end{equation}
We are now in a position to use \eqref{eq:secondCarleman} on $v$. Taking into account \eqref{eq:conditionfirstofthev}, \eqref{eq:conditionsecondofthev}, we have for all $\tau\geq \tau_0$, $0<h\leq h_0$, $sh\leq \epsilon_0$
\begin{equation*}
\begin{split}
&J(v)+\E\int_{Q}s^{3}\lambda^{4}_{1}\varphi^{3}\,e^{2s\varphi}|v|^{2}\,dt+\sum_{i,j=1}^{n}\E\int_{Q_{ij}^\ast}s^{-1}e^{2s\varphi}|D_{ij}^2v|^2\, dt+\E\int_{Q}s\lambda_{1}^2\varphi\,e^{2s\varphi}|a_3 v|^2dt\\
&\leq\,\mathcal{C}_{\tau_{0},\varepsilon_0}\left(\E\int_{Q}\,e^{2s\varphi}\left|\sum_{i=1}^{n}a_{1i}A_iD_i(v)+a_2v+\mathbf{F}\right|^2\,dt+\sum_{i=1}^{n}\E\int_{Q_i^{\ast}}\,e^{2s\varphi}|D_i(a_3v)|^2dt\right.\\
&\left.+\sum_{i=1}^{n}\E\int_0^T\int_{\Gamma}st_{r}^{i}(e^{2s\varphi})|t_{r}^{i}(D_iv)|^2\,dt\right)+\mathcal{C}\tau^3e^{\tau\mathcal{C}}\|\xi\|^2_{H^1_{\mathcal{F}}(0,T;H_h^1(\Gamma))},
\end{split}
\end{equation*}

Using the Young inequality and \eqref{eq:difference:product} to the first and second integral from the right-hand side above, respectively, we can bound the above inequality as
\begin{equation}\label{eq:proofsecondresultdCarleman}
\begin{split}
J(v)+&\E\int_{Q}s^{3}\lambda^{4}_{1}\varphi^{3}\,e^{2s\varphi}|v|^{2}\,dt+\sum_{i,j=1}^{n}\E\int_{Q_{ij}^{\ast}}s^{-1}e^{2s\varphi}|D_{ij}^2v|^2\, dt+\E\int_{Q}s\lambda_{1}^2\varphi\,e^{2s\varphi}|a_3 v|^2dt\\
&\leq\,\mathcal{C}\tau^3e^{\tau\mathcal{C}}\|\xi\|^2_{H^1_{\mathcal{F}}(0,T;H_h^1(\Gamma)}+\mathcal{C}\left(\E\int_{Q}\,e^{2s\varphi}|\mathbf{F}|^2dt+\sum_{i=1}^{n}\E\int_{Q}\,e^{2s\varphi}|A_{i}D_{i}v|^2\,dt\right.\\
&\left.+\E\int_{Q}e^{2s\varphi}|a_2v|^2\,dt+\sum_{i=1}^{n}\E\int_{Q_i^{\ast}}\,e^{2s\varphi}\left(|A_{i}v|^2+|D_{i}v|^2\right)\,dt\right.\\
&\left.+\sum_{i=1}^{n}\E\int_0^T\int_{\Gamma}st_{r}^{i}(e^{2s\varphi})|t_{r}^{i}(D_iv)|^2\,dt\right).
\end{split}
\end{equation}
Noting that using \eqref{eq:inequalityAverange}, \eqref{eq:int:ave} and $e^{-2s\varphi}A_i(e^{2s\varphi})=1+\mathcal{O}((sh)^2)$, for each $i{ =1,...,n.}$ we have
\begin{equation*}
\begin{split}
\E\int_{Q_i^{\ast}}e^{2s\varphi}|A_{i}v|^2\,dt\leq& \E\int_{Q_i^{\ast}}e^{2s\varphi}A_i(|v|^2)\,dt=\; \E\int_{Q}A_i(e^{2s\varphi})|v|^2\,dt+\frac{h}{2}\E\int_{\partial_i Q} |v|^2t_r^{i}(e^{2s\varphi})\,dt\\
\leq&\E\int_{Q}e^{2s\varphi}|v|^2\,dt+\frac{h}{2}\E\int_{\partial_i Q} |v|^2t_r^{i}(e^{2s\varphi})\,dt+\mathcal{O}((sh)^2)\int_{Q}e^{2s\varphi}|v|^2\,dt
\end{split}
\end{equation*}
Therefore, combining the above inequalities and \eqref{eq:proofsecondresultdCarleman}, we obtain the following  
\begin{equation*}
\begin{split}
&J(v)+\E\int_{Q}s^{3}\lambda^{4}_{1}\varphi^{3}\,e^{2s\varphi}|v|^{2}\,dt+\sum_{i,j=1}^{n}\E\int_{Q_{ij}^{\ast}}s^{-1}e^{2s\varphi}|D_{ij}^2v|^2\, dt+\E\int_{Q}s\lambda_{1}^2\varphi\,e^{2s\varphi}|a_3 v|^2dt\\
&\leq\,\mathcal{C}\left(\E\int_{Q}\,e^{2s\varphi}|\mathbf{F}|^2dt+\E\int_{Q}\,e^{2s\varphi}|v|^2\,dt+\mathcal{O}((sh)^2)\int_{Q}e^{2s\varphi}|v|^2\,dt\right.\\
&+\sum_{i=1}^{n}\left(\E\int_{Q}\,e^{2s\varphi}|A_{i}D_{i}v|^2
+\frac{h}{2}\int_0^T\int_{\Gamma} |v|^2t_r^{i}(e^{2s\varphi})+\E\int_{Q_i^{\ast}}\,e^{2s\varphi}|D_{i}v|^2\right)\,dt\\
&\left.+\sum_{i=1}^{n}\E\int_0^T\int_{\Gamma}st_{r}^{i}(e^{2s\varphi})|t_{r}^{i}(D_iv)|^2\,dt+\E\int_{Q}e^{2s\varphi}|v|^2\,dt\right)+\mathcal{C}\tau^3e^{\tau\mathcal{C}}\|\xi\|^2_{H^1_{\mathcal{F}}(0,T;H_h^1(\Gamma))}.
\end{split}
\end{equation*}

Recalling that $s(t)=\tau\theta(t)$ and letting $\tau$ sufficiently large and taking a convenient constant, we notice that the second, third, and fifth integrals are absorbed by the left side in the above inequality, and the fourth integral is controlled by the last integral on the right side. Thus, the above inequality can be rewritten as follows
\begin{equation}\label{eq:absorbthecarleman}
\begin{split}
&J(v)+\E\int_{Q}s^{3}\lambda^{4}_{1}\varphi^{3}\,e^{2s\varphi}|v|^{2}\,dt+\sum_{i,j=1}^{n}\E\int_{Q_{ij}^{\ast}}s^{-1}e^{2s\varphi}|D_{ij}^2v|^2\, dt+\E\int_{Q}s\lambda_{1}^2\varphi\,e^{2s\varphi}|a_3 v|^2dt\\
&\leq\,\mathcal{C}\E\int_{Q}\,e^{2s\varphi}|\mathbf{F}|^2dt+\Tilde{\mathcal{C}}\tau^3e^{\tau\mathcal{C}}\left(\|{ \partial_{\nu}w}\|^2_{L^2_\mathcal{F}(0,T:L_h^2(\Gamma))}+\|\xi\|^2_{H^1_{\mathcal{F}}(0,T;H_h^1(\Gamma))}\right).
\end{split}
\end{equation}
Now, by \eqref{eq:propertyonewight}, we have $\overline{\mathcal{M}}_0\times \left(t_0-\frac{\varepsilon}{\sqrt{\Tilde{N}}},t_0+\frac{\varepsilon}{\sqrt{\Tilde{N}}}\right)\subset G^{(4)}\cap Q$ and then, observing \eqref{eq:definitionchi}, we see that
\begin{equation}\label{eq:estimateoftheleftside}
    \begin{split} &J(v)+\E\int_{Q}s^{3}\lambda^{4}_{1}\varphi^{3}\,e^{2s\varphi}|v|^{2}\,dt+\sum_{i,j=1}^{n}\E\int_{Q_{ij}^\ast}s^{-1}e^{2s\varphi}|D_{ij}^2v|^2\, dt+\E\int_{Q}s\lambda_{1}^2\varphi\,e^{2s\varphi}|a_3 v|^2dt\\
     &\geq   \E\int_{Q}s^{3}\lambda^{4}_{1}\varphi^{3}\,e^{2s\varphi}|v|^{2}\,dt+\sum_{i=1}^{n}\E\int_{Q_{i}^{\ast}}s\varphi e^{2s\theta\varphi}\,|D_{i}v|^{2}\,dt+\sum_{i,j=1}^{n}\E\int_{Q_{ij}^{\ast}}s^{-1}e^{2s\varphi}|D^2_{ij}v|^2\,dt\\
     &\geq \mathcal{C}\tau^{-1}e^{2\mu_4 \tau}\int_{t_0-\frac{\varepsilon}{\sqrt{\Tilde{N}}}}^{t_0+\frac{\varepsilon}{\sqrt{\Tilde{N}}}}\left(\E\int_{\mathcal{M}_0}|w|^{2}+\sum_{i=1}^{n}\E\int_{(\mathcal{M}_{0})_{i}^{\ast}}\,|D_{i}w|^{2}+\sum_{i,j=1}^{n}\E\int_{(\mathcal{M}_0)_{ij}^\ast}|D^2_i w|^2\right)\,dt.
    \end{split}
\end{equation}

On the other hand, we can verify that there exists a sufficiently small $h_1$ such that
\begin{equation}\label{propchih}
\partial_{t}\chi=A_iD_i\chi=D_i^{2}\chi=0\quad (x,t)\in Q^{'}
\end{equation}
for all $h\in (0,h_1)$, where
\begin{equation*}
    Q^{'}:=\left\{(x,t)\in Q\;:\quad \varphi(x)\theta(t)>\mu_3+\frac{1}{2}(\mu_4-\mu_3)\right\}.
\end{equation*}
We apply Proposition \ref{pro:product} and \eqref{propchih} to yield 
\begin{equation}\label{eq:estimateoftheF}
    \begin{split}
&\E\int_{Q}e^{2s\varphi}|\mathbf{F}|^2\,dt\\
&\leq \mathcal{C}e^{2\tau\left(\mu_3+\frac{1}{2}(\mu_4-\mu_3)\right)}\E\int_{Q\setminus Q^{'}}\left(|w|^2+\sum_{i=1}^{n}\left[\left(1+\frac{h^4}{16}\right)|A_iD_iw|^2+\frac{h^4}{16}|D_i^2w|^2\right]\right)\,dt\\
&\leq \mathcal{C}e^{\tau (\mu_4+\mu_3)}\|w\|^2_{L_{\mathcal{F}}^2(0,T;H_h^2(\mathcal{M}))},
\end{split}
\end{equation}
where in the second term of the second line, we have used that
\begin{equation*}
    \E\int_{Q}|A_iD_iw|^2\,dt\leq \E\int_{Q}A_i(|D_iw|^2)\,dt\leq \E\int_{Q_i^{\ast}}|D_iw|^2\,dt,
\end{equation*}
which follows after the combination of \eqref{eq:inequalityAverange} and \eqref{eq:int:ave}. Therefore, from \eqref{eq:absorbthecarleman}, \eqref{eq:estimateoftheleftside} and \eqref{eq:estimateoftheF} it follows that
\begin{equation*}
\begin{split}
     \mathcal{C}\tau^{-1}e^{2\mu_4 \tau}&\E\int_{t_0-\frac{\varepsilon}{\sqrt{\Tilde{N}}}}^{t_0+\frac{\varepsilon}{\sqrt{\Tilde{N}}}}\left(\E\int_{\mathcal{M}_0}|w|^{2}+\sum_{i=1}^{n}\E\int_{(\mathcal{M}_{0})_{i}^{\ast}}\,|D_{i}w|^{2}+\sum_{i,j=1}^{n}\E\int_{(\mathcal{M}_0)^{\ast}_{ij}}|D^2_{ij} w|^2\right)\,dt\\
     &\leq\, \mathcal{C}e^{\tau (\mu_4+\mu_3)}\|w\|_{L_{\mathcal{F}}^2(0,T;H_h^2(\mathcal{M}))}+\Tilde{\mathcal{C}}\tau^3e^{\tau\mathcal{C}}\left(\|{ \partial_{\nu}w}\|^2_{L^2_\mathcal{F}(0,T:L_h^2(\Gamma))}+\|\xi\|^2_{H^1_{\mathcal{F}}(0,T;H_h^1(\Gamma))}\right),
\end{split}
\end{equation*}
which implies 
\begin{equation}\label{eq:endproof}
\begin{split}
     &\E\int_{t_0-\frac{\varepsilon}{\sqrt{\Tilde{N}}}}^{t_0+\frac{\varepsilon}{\sqrt{\Tilde{N}}}}\left(\E\int_{\mathcal{M}_0}|w|^{2}+\sum_{i=1}^{n}\E\int_{(\mathcal{M}_{0})_{i}^{\ast}}\,|D_{i}w|^{2}+\sum_{i,j=1}^{n}\E\int_{(\mathcal{M}_0)^{\ast}_{ij}}|D^2_{ij} w|^2\right)\,\,dt\\
     &\leq\, \mathcal{C}\tau e^{-\tau (\mu_4-\mu_3)}\|w\|^2_{L_{\mathcal{F}}^2(0,T;H_h^2(\mathcal{M}))}+\mathcal{C}\tau^4 e^{\mathcal{C}\tau}\left(\|{ \partial_{\nu}w}\|^2_{L^2_\mathcal{F}(0,T:L_h^2(\Gamma))}+\|\xi\|^2_{H^1_{\mathcal{F}}(0,T;H_h^1(\Gamma))}\right)\\
     &\leq  e^{-\frac{1}{2}\tau (\mu_4-\mu_3)}\|w\|^2_{L_{\mathcal{F}}^2(0,T;H_h^2(\mathcal{M}))}+ e^{2\mathcal{C}\tau}\left(\|{ \partial_{\nu}w}\|^2_{L^2_\mathcal{F}(0,T:L_h^2(\Gamma))}+\|\xi\|^2_{H^1_{\mathcal{F}}(0,T;H_h^1(\Gamma))}\right),
\end{split}
\end{equation}
where we have used that there exists a sufficiently large $\tau_1$ such that
\begin{equation*}
    \mathcal{C}\tau e^{-\tau(\mu_4-\mu_3)}\leq e^{-\frac{1}{2}\tau(\mu_4-\mu_3)}
\end{equation*}
for all $\tau_1<\tau <C\varepsilon_0h^{-1}$. 

Consider $f(\tau)=ae^{-\frac{1}{2}\tau(\mu_4-\mu_3)}+be^{2\mathcal{C}\tau}$ with $\tau_1 <\tau<\mathcal{C}\varepsilon_0h^{-1}$, $a=\|w\|^2_{L^2_{\mathcal{F}}{(0,T;H_h^2(\mathcal{M}))}}$ and $b=\|{ \partial_{\nu}w}\|^2_{L^2_{\mathcal{F}}(0,T;L_h^2(\Gamma))}+\|\xi\|^2_{H^1_{\mathcal{F}}(0,T;H_h^1(\Gamma))}.$ As $\mu_4-\mu_3>0$, then from Lemma 3.39 in \cite{LOPD:2023} the function $f$ satisfies the following.
\[
\min_{\tau\in [\tau_1,C\varepsilon_0h^{-1}]}f(\tau)\leq 2\max\left\{be^{2\mathcal{C}\tau_1},ae^{-\frac{1}{2h}C\varepsilon_0(\mu_4-\mu_3)},ae^{-\frac{1}{2}\tau_{\ast}(\mu_4-\mu_3)}\right\},
\]
for any $\displaystyle \tau_{\ast}\leq \frac{2}{4\mathcal{C}+(\mu_4-\mu_3)}\ln{\left(\frac{a}{b}\right)} $. Given that \eqref{eq:endproof} is valid for all $\tau \in [\tau_1,\mathcal{C}\varepsilon_0h^{-1}]$ and taking $\tau_\ast=\frac{2}{4\mathcal{C}+(\mu_4-\mu_3)}\ln{\left(\frac{a}{b}\right)}$, it follows in particular that: 
\begin{equation}\label{eq:Estimationfinally}
\begin{split}
    \|w\|^2_{L^2_{\mathcal{F}}\left((t_0-\frac{\varepsilon}{\sqrt{\Tilde{N}}},t_0+\frac{\varepsilon}{\sqrt{\Tilde{N}}};H_h^2(\mathcal{M}_0)\right)}\leq\; 2\max\left\{be^{2\mathcal{C}\tau_1},ae^{-\frac{1}{2h}\mathcal{C}\varepsilon_0(\mu_4-\mu_3)},a^{\kappa}b^{1-\kappa}\right\}
\end{split}
\end{equation}
with
\begin{equation*}
    \kappa:=\frac{4\mathcal{C}}{4\mathcal{C}+(\mu_4-\mu_3)}\in (0,1)
\end{equation*}
By the argument, we see that the constant of the right-hand side of \eqref{eq:Estimationfinally} is also independent of $t_0$. 

Since $t_0 \in [\sqrt{2}\varepsilon,\, T-\sqrt{2}\varepsilon]$, in \eqref{eq:Estimationfinally} we choose the points $t_0^{\,j} = \sqrt{2}\varepsilon + \frac{j\varepsilon}{\sqrt{\tilde{N}}}$ for $j = 0,1,2,\ldots,m$, where $m$ is chosen such that
\[
\sqrt{2}\varepsilon + \frac{m\varepsilon}{\sqrt{\tilde{N}}}
\le T-\sqrt{2}\varepsilon
\le \sqrt{2}\varepsilon + \frac{(m+1)\varepsilon}{\sqrt{\tilde{N}}}
\le T.
\]
Summing the corresponding estimates with respect to $j$ from $0$ to $m$, and then replacing $\varepsilon$ by $\sqrt{2}\varepsilon$ (which is admissible since $\varepsilon>0$ is arbitrary), we obtain the desired inequality \eqref{eq:teo:secondresult}, and this completes the proof of Theorem \ref{Teo:SecondResult}.
\section{Comments and concluding remarks}
{ 
In Theorem 1.1, we establish a Lipschitz stability result for an inverse source problem for semi-discrete stochastic parabolic equations in arbitrary spatial dimensions. The stability is obtained from interior observations together with the terminal value of the solution. In contrast with [20], the observation region is an arbitrary open subset of the spatial domain. Moreover, the regularity assumption on the coefficient $a_2$ is weaker in comparison with the existence literature. For this reasons, this problem is new. In Theorem 1.3, we prove a Hölder stability result for a Cauchy problem using the terminal value of the solution and the trace of the spatial derivatives on an arbitrary open subset of the boundary over a time interval. The novelty here lies in the treatment of arbitrary spatial dimensions and in allowing the boundary observation set to be a arbitrary open subset. From a methodological point of view, we derive three new Carleman estimates adapted to the different settings considered in the paper (see Theorems 1.5, 3.3 and 3.7). These estimates require new weight functions and refined analytical arguments in the discrete stochastic framework. In particular, they allow us to establish observability inequalities on arbitrary open subsets, which is not addressed in [20]. See, for instance, assumption 3.1 and the related discussion. Moreover, Theorem 3.3 and 3.7 generalize the results presented in [20] since the measurement on the boundary are arbitrary open set sampled on the mesh. 

Several variants of the inverse problem are related to the results presented in this manuscript. In particular, an interesting open problem is to obtain a stability result similar to Theorem~\ref{firstresulttheorem}, but for a drift term $f$ in system~\ref{systemrandomsource}, or, in a more general case, for both the drift and diffusion terms at the same time. In this direction, we can refer to the uniqueness results in~\cite{LZ:2024,QLu:2012} obtained in a continuous setting, where one of the main tools is the use of Carleman estimates. Although a similar analysis could be considered in our case, the arguments cannot be applied directly to the semi-discrete framework. In fact, besides Carleman estimates, some additional tools are needed in the continuous case, and it is not clear how to adapt them to the semi-discrete setting. This point is important because uniqueness results for inverse problems are not always true in the semi-discrete context. For example, Theorem~\ref{firstresulttheorem} implies a uniqueness result, while Remark~\ref{remark:nonuniquess} shows that uniqueness cannot be obtained from the corresponding stability result. This difference shows that uniqueness is a delicate issue in the semi-discrete case and makes the study of the results in~\cite{LZ:2024,QLu:2012} under this framework interesting.

Furthermore, the result presented in this work could also be applied to extend the recent result provided in \cite{WZZW:2025}. Indeed, in \cite{WZZW:2025} the authors established Lipschitz stability for an inverse problem of a fully-discrete stochastic hyperbolic equation in 1-D spatial dimension; and their results are based in a new Carleman estimate. 
}
\section*{Acknowledgments}
 R. Lecaros was partially supported by FONDECYT (Chile) Grant 1221892 and Proyecto Interno USM 2025 PI$\_$LIR$\_$25$\_$14. A. A. P\'erez acknowledges the support of Vicerrector\'ia de Investigaci\'on y postgrado, Universidad del B\'io-B\'io, project IN2450902 and FONDECYT Grant 11250805. M. F. Prado gratefully acknowledges the support of the Institutional Scholarship Fund of the Universidad de Valparaiso (FIB-UV) and Proyecto Interno USM 2025 PI$\_$LIR$\_$25$\_$14.
\appendix
\section{Proof of semi-discrete Carleman estimate with homogeneous Dirichlet condition (Theorem \ref{theo:Carleman_datainterior}).}\label{Proof:CarlemanEstimateHomegeneuos}
In this section, we prove the semi-discrete Carleman estimate under homogeneous Dirichlet conditions. The argument follows a classical approach (see \cite{fursikov-1996}), based on conjugating the original operator with a suitably chosen exponential weight. Our estimate is closely related to the one presented in \cite{boyer-2010-1d-elliptic}, and we adopt its methodology whenever possible. The main difference is that we do not impose any boundary condition on the state or on the weight function, which leads to additional boundary terms that will be handled later on.

The proof is divided into three steps. First, in Section \ref{sec:conjugatedoperator}, we decompose the conjugated operator into its symmetric part $P_{1}$, antisymmetric part $P_{2}$, and an auxiliary term $R$, which helps simplify the calculations. Next, in Section \ref{sec:crossProduct}, we estimate the cross-inner products between these operators. Finally, in Section \ref{sec:returninbackvarible}, we return to the original variable to conclude the argument.

\subsection{Conjugated operator}\label{sec:conjugatedoperator}
For simplicity of notation, we write $r:=e^{s\varphi}$ and $\rho:=r^{-1}$, where our weight function is defined as $\varphi=e^{\lambda\psi}$, for $s\geq 1$, with $\psi\in C^{k}$ for $k$ sufficiently large and $\lambda\geq 1$. The proof of some technical results can be found in Section 5.  \\
For all $i{ =1,...,n.}$, let us consider the functions $\gamma_{i}$ such that $\mbox{reg}(\gamma)<\mbox{reg}^{0}$ and the following notation
\begin{equation*}
\nabla_{\gamma}f:=(\sqrt{\gamma_1}D_1 f_1,\cdots,\sqrt{\gamma_n}D_nf_n)\quad\text{and}\,\quad \Delta_{\gamma}f:=\sum_{i=1}^{n}\gamma_i \partial_{x_i}^{2}f. 
\end{equation*}
\par Let $\mathcal{P}(w):=\,dw-\sum_{i=1}^{n}D_{i}\left(\gamma_iD_{i}w\right)dt=fdt+gdB(t)$. With respect to this weight function, the first step is to consider the change of variable $w=\rho z$. Our first task is to split the conjugate operator 
$$r\mathcal{P}(\rho z)=rd(\rho z)-r\sum_{i=1}^{n}D_{i}\left(\gamma_{i}D_{i}(\rho z)\right)dt$$
into simple terms that we will estimate separately. Using It\^{o}'s formula  (cf. \citep[Theorem 4.18]{klebaner2012introduction}) and notice that $r\rho=1$, we obtain the following result
\begin{equation*}
    rd(\rho z)=dz+r\partial_t(\rho) z\,dt
\end{equation*}
 By \eqref{eq:difference:product} from Proposition \ref{pro:product}, we have
\begin{align*}
    D_{i}\left( \gamma_{i}D_{i}(\rho z)\right)
    =&D_{i}(\gamma_{i}D_{i}\rho)A_{i}^{2}z+A_{i}(\gamma_{i}D_{i}\rho)D_{i}A_{i}z+D_{i}A_{i}\rho A_{i}(\gamma_{i}D_{i}z)+A_{i}^{2}\rho D_{i}(\gamma_{i}D_{i}z).
\end{align*}
Then, applying Proposition \ref{pro:product} and using $A_{i}\gamma_{i}=\gamma_{i}+h\mathcal{O}(1)$, we have
\begin{align*}
    A_{i}(\gamma_{i}D_{i}z)&=A_{i}\gamma_{i}A_{i}D_{i}z+\frac{h^2}{4}D_{i}\gamma_{i}D_{i}^{2}z=\gamma_i\, A_iD_iz+\frac{h^2}{4}D_i\gamma_i\,D^2_iz+h\mathcal{O}(1)\,A_iD_iz,\\
    A_{i}(\gamma_{i}D_{i}\rho)&=A_{i}\gamma_{i}A_{i}D_{i}\rho+\frac{h^2}{4}D_{i}\gamma_{i}D_{i}^{2}\rho=\gamma_i\, A_iD_i\rho+\frac{h^2}{4}D_i\gamma_i\,D_i^2\rho+h\mathcal{O}(1)\,A_iD_i\rho,\\
    D_{i}(\gamma_{i}D_{i}\rho)&=D_{i}\gamma_{i}A_{i}D_{i}\rho+A_{i}\gamma_{i}D_{i}^{2}\rho=D_i\gamma_i\,A_iD_i\rho + \gamma_i\, D_i^2\rho+h\mathcal{O}(1)\,D_i^2\rho.
\end{align*}
Thus, the conjugate operator can be written as
 \begin{equation}\label{eq:conjugate}
\begin{split}
rd(\rho z)-r\sum_{i=1}^{n}D_{i}\left(\gamma_{i}D_{i}(\rho z)\right)dt+R_{h}z\,dt=&B_1z+B_2z\,dt+C_1z\,dt+C_2z\,dt+C_3z\,dt,
\end{split}
\end{equation}
where 
\begin{equation*}
    \begin{split}
        C_{1}z:=&-\sum_{i=1}^{n}rA_{i}^{2}\rho\,D_{i}(\gamma_{i}D_{i}z),\ C_{2}z:=-\sum_{i=1}^{n}\gamma_{i}rD_{i}^{2}\rho A_{i}^{2}z,\ C_{3}z:=r\partial_t(\rho) z
    \end{split}
\end{equation*}
\begin{equation*}
    \begin{split}
        B_1z:=& dz,\ B_2z=-2\sum_{i=1}^{n}\gamma_{i}rD_{i}A_{i}\rho D_{i}A_{i}z;
    \end{split}
\end{equation*}
and
\begin{equation*}
    \begin{split}
        R_{h}z:=\,&\sum_{i=1}^{n}\left(h\mathcal{O}(1)rD_{i}^{2}\rho +D_{i}\gamma_{i}rA_{i}D_{i}\rho\right)A_{i}^{2}z+\sum_{i=1}^{n}h_{i}\mathcal{O}(1)rD_{i}A_{i}\rho D_{i}A_{i}z\\
        &+\sum_{i=1}^{n}\frac{h^{2}}{4}D_{i}\gamma_{i}rD_{i}^{2}\rho D_{i}A_{i}z+\sum_{i=1}^{n}\frac{h^{2}}{4}D_{i}\gamma_i \,rD_{i}A_{i}\rho D_{i}^{2}z.
    \end{split} 
\end{equation*}
 Therefore, adding the terms $$C_4z:=\,-\frac{h^2}{4}\sum_{i=1}^{n}D_i(\gamma_i D_i(rD_i^2\rho)A_iz),\quad C_5z:=\,-\frac{h^2}{4}\sum_{i=1}^{n}D_i(D_i(\gamma_irD^2_i\rho)A_iz)$$ 
  $$ C_{6}z=-B_3z:=\,-2s\Delta_{\gamma}\varphi\, z$$
 in \eqref{eq:conjugate} we have the following identity in $Q$
\begin{equation}\label{eq:opconjugate}
    r\mathcal{P}(\rho z)=Cz\,dt+Bz-M_{h} z\,dt,
\end{equation}
where $Cz=C_{1}z+C_{2}z+C_3z+C_4z+C_5z+C_6z$, $Bz:=B_1z+(B_{2}z+B_3z)
\,dt$ and $M_{h}z=C_4z+C_5z+R_hz$.\\
Thus, our next task is to estimate the cross-product $\displaystyle 2\E \int_Q CzBz =2\E \sum_{i=1}^{6}\sum_{j=1}^{3}\int_Q C_{i}zB_{j}z$ and to bound the left-hand side of \eqref{eq:opconjugate} with respect to the term $M_h z$. Indeed, 
We notice that from \eqref{eq:opconjugate}, it holds that
\begin{equation}\label{eq:estimationP(q)}
    2r\mathcal{P}(q)\,Cz=2|Cz|^2dt+2Cz\,Bz-2M_hz\,Cz\,dt.
\end{equation}
Hence, combining the definition of $r\mathcal{P}(q)$ with \eqref{eq:estimationP(q)}, it follows that
\begin{equation*}
    \E \int_Q 2rf\,Cz\,dt=2\E \int_Q |Cz|^2dt+2\E \int_{Q}CzBz-2\E\int_{Q}
    M_hz\,Cz\,dt.
\end{equation*}
Using the Young inequality, we notice that
\begin{equation}\label{eq:Estirf}
    2\E \int_0^T |rf|^2\,dt \geq \E\int_{Q}|Cz|^2dt+ 2\E\int_{Q} Cz\,Bz -2\E \int_{Q}|M_{h}z|^2\,dt.
\end{equation}
The next step is to provide an estimate for the right-hand side of \eqref{eq:Estirf}.
\begin{remark}\label{remark:change}
The decomposition of \eqref{eq:opconjugate} differs from that presented in \cite{LPP:2025}, not only due to the sign change, but also because the corrector, which is typically added when the dimension is greater than one, is included as an additional term rather than as a component of $M_hz$. As a result, the estimate for the cross-product includes an additional term that must be considered.
Moreover, in the estimate \eqref{eq:Estirf}, Young's inequality is applied in a way that preserves the term $|C(z)|^2$, which will play a crucial role in the final estimate of the Carleman inequality. These are some of the reasons for the study of this new configuration in order to obtain the Carleman inequality.
\end{remark}
\subsection{An estimate for the cross-product}\label{sec:crossProduct}
Recall that 
\begin{equation}\label{conmu}
    2\E \int_Q CzBz =2\E \sum_{i=1}^{6}\sum_{j=1}^{3}\int_Q C_{i}zB_{j}z:=\,\sum_{i=1}^{6}\sum_{j=1}^{3}I_{ij}. 
\end{equation}

The procedure for estimating the product \eqref{conmu} is analogous to that in \cite{LPP:2025}, except for a change of sign and consideration of Remark \ref{remark:change}. Consequently, we divide the estimate into three groups: terms involving the differential $dz$, products of the additional terms, and terms involving the differential $dt$.

We begin with the terms involving the differential $dt$. Specifically, consider the terms $I_i$ for $i=1,2,3,\dots,5$, whose estimates can be found in \cite{LPP:2025}, up to a sign change. For the additional term $I_6$, which does not appear in \cite{LPP:2025}, its estimate is provided in Appendix \ref{sec:proof:carleman}. Consequently, we obtain the following preliminary estimate:
\begin{lemma}\label{lem:dz}
(\textit{Terms involving differential $dz$.}) For $\displaystyle \max_{t\in[0,T]}\{\theta(t)\}\tau h\leq 1$, we have
\begin{equation*}
\begin{split}
    \sum_{i=1}^{6}I_{i1}\geq\,&-\sum_{i=1}^{n} \E\int_{Q_i^{\ast}}\gamma_i|D_i(dz)|^2+\E\int_{Q}(s^2\lambda^2\varphi^2+s\lambda^2\varphi)|\nabla_{\gamma} \psi|^2|dz|^2-X_{1}-Y_{1}
\end{split}
\end{equation*}
where
\begin{equation*}
\begin{split}
    X_{1}:=\,&\sum_{i=1}^{n}\left(\E\int_{Q_{i}^{\ast}}\mathcal{O}_{\lambda}((sh)^2)\,|D_i(dz)|^2+\E\int_{Q_{i}^{\ast}}T\theta(t)\mathcal{O}_{\lambda}((sh)^2)\,|D_iz|^2\,dt\right)\\
    &+\E\int_{Q}(\mathcal{O}_{\lambda}((sh)^2)+T\mathcal{O}_{\lambda}(1)+s\lambda\varphi\mathcal{O}(1))\,|dz|^2\\
    &+\E\int_{Q} T(\theta\mathcal{O}_{\lambda}((sh)^2)-2s\varphi(\lambda^3\beta|\nabla_{\gamma}\psi|^2+\lambda^2\mathcal{O}(1)))+\mathcal{O}_{\lambda}(1))\,|z|^2dt
\end{split}
\end{equation*}
and
\begin{equation*}
\begin{split}
    Y_{1}:=\,&-\left.\E\int_{\mathcal{M}}[s^2\mathcal{O}_{\lambda}(1)+2s(\lambda^2\varphi|\nabla_{\gamma}\psi|^2+\lambda\varphi\mathcal{O}(1))]|z|^2\right|_0^T+\sum_{i=1}^{n}\left.\E\int_{\mathcal{M}_{i}}\gamma_i|D_iz|^2\right|_0^T\\
    &\sum_{i=1}^{n}\E\left. \int_{\mathcal{M}_i^*}\mathcal{O}_{\lambda}((sh)^2)\,|D_iz|^2\right|_0^T+\E\left.\int_{\mathcal{M}}(\mathcal{O}_{\lambda}((sh)^2)+T\theta^2\mathcal{O}(1))\,|z|^2\right|_0^T.
\end{split}
\end{equation*}
\end{lemma}
Now, let us consider the products of the additional terms. In this case, the new terms $I_{62}$ and $I_{63}$ appear, whose estimates are provided in Appendix \ref{sec:proof:carleman}. For the remaining terms, the estimates remain unchanged compared to \cite{LPP:2025}, since the weight functions still satisfy the same properties, namely $|\theta(t)| \leq C$ and $|\theta_t(t)| \leq C$ for some constant $C>0$ and all $t \in [0,T]$. Therefore, we obtain the following estimate:
\begin{lemma}\label{lem:AdditionalsTerms}
    (\textit{product of the additional terms.}) For $\displaystyle \max_{t\in[0,T]}\{\theta(t)\}\tau h\leq 1$, we obtain
    \begin{equation*}
        \sum_{i=4}^{6}I_{i2}+\sum_{i=1}^{6}I_{i3}\geq\, \sum_{i=1}^{n}\E\int_{Q_{i}^{\ast}}4s\lambda^2\varphi\gamma_i|\nabla_{\gamma}\psi|^2\,|D_iz|^2\,dt-\E\int_{Q}4s^3\lambda^4\varphi^3|\nabla_{\gamma}\psi|^4\,|z|^2\,dt-X_3
    \end{equation*}
where
\begin{align*}
        &X_3:=\, \E\int_{Q} s|\mathcal{O}_{\lambda}((sh)^2)|\,|z|^2\,dt-E\int_{Q}s\mathcal{O}_{\lambda}((sh)^2)\,|z|^2\,dt+\sum_{i=1}^{n} \E\int_{Q_i^{\ast}}s|\mathcal{O}_{\lambda}((sh)^2)|\,|D_iz|^2\,dt\\
        &+\sum_{i=1}^{n} \E\int_{Q_i^{\ast}}s\mathcal{O}_{\lambda}((sh)^2)\,|D_iz|^2\,dt +\sum_{i=1}^{n}\E\int_{Q_i^{\ast}}(4s\lambda\Delta_{\gamma}\psi\, \varphi \gamma_i+s\mathcal{O}_{\lambda}(h+(sh)^2))\,|D_iz|^2\,dt\\
        &+\E\int_{Q}s\mathcal{O}_{\lambda}(1)\,|z|^2\,dt+\E\int_{Q}(s^3\lambda^3\varphi^3\mathcal{O}(1)+s^2\mathcal{O}_{\lambda}(1)+s^3\mathcal{O}_{\lambda}((sh)^2))\,|z|^2\,dt\\
        &+\E\int_{Q} (Cs\lambda^2T\theta^2\varphi|\nabla_{\gamma}\psi|^2+s\lambda T\theta^2\varphi\mathcal{O}(1))\,|z|^2\,dt {-2\E\int_{Q}s^2\lambda^4\varphi^2|\nabla_{\gamma}\psi|^4|z|^2\,dt}\\
        &{+2\E\int_{Q}(s^2\lambda^3\varphi^2|\nabla_{\gamma}\psi|^4-4s^2\lambda^2|\mathcal{O}(1)|^2-s\mathcal{O}_{\lambda}(1)((sh)^2))\,|z|^2\,dt}.
\end{align*}
\end{lemma}
Finally, the terms $I_{12}$, $I_{22}$ and $I_{32}$ behave analogously to those in the deterministic setting studied in \cite{boyer-2014}, as the time variable does not play a significant role. For this reason, we omit a detailed proof of the following lemma, since the only difference with respect to \cite{boyer-2014} lies in our notation.
\begin{lemma}\label{lem:dt}
(\textit{Terms involving differential $dt$.}) For $\displaystyle \max_{t\in[0,T]}\{\theta(t)\}\tau h\leq 1$, we have the following inequality
\begin{equation*}
    \sum_{i=1}^{3}I_{i2}\geq\,\E\int_{Q} 6s^3\lambda^4\varphi^3|\nabla_{\gamma}\psi|^4\,|z|^2\,dt-\sum_{i=1}^{n}\E\int_{Q_{i}^{\ast}}2s\lambda^2\varphi \gamma_{i}|\nabla_{\gamma}\psi|^2\,|D_iz|^2\,dt-X_2-Y_2
\end{equation*}
where
\begin{align*}
&X_{2}:=\,\sum_{i=1}^{n}\E\int_{Q}|s\lambda\varphi\mathcal{O}(1)+\mathcal{O}_{\lambda}(sh)+s\mathcal{O}_{\lambda}(sh)+s\mathcal{O}_{\lambda}((sh)^{2})|\,|D_{i}A_{i}z|^{2}\,dt\\
    &+\sum_{i=1}^{n}\E\int_{Q}h^{2}|\mathcal{O}_{\lambda}(sh)|\,|D_{i}^{2}z|^{2}\,dt+\sum_{i\ne j=1}^{n}\E\int_{Q_{ij}^{*}}\left|h\lambda\mathcal{O}(sh)+h\mathcal{O}_{\lambda}((sh)^{2})\right|\,|D_{ij}^{2}z|^{2}\,dt\\
    &+\sum_{i=1}^{n}\E\int_{Q^{\ast}_{i}}|s\lambda\varphi\mathcal{O}(1)+s\mathcal{O}_{\lambda}(sh)+h\mathcal{O}_{\lambda}(sh)+s\mathcal{O}_{\lambda}((sh)^{2})+\mathcal{O}_{\lambda}((sh)^2)|\, |D_{i}z|^{2}\\
    &+\E\int_{Q}(s^2\lambda^3\varphi^2\mathcal{O}(1)+s^2\mathcal{O}_{\lambda}(1)+s^2T\theta\mathcal{O}_{\lambda}(1)+s^3\mathcal{O}_{\lambda}((sh)^2))|z|^2\\
    &+\sum_{i=1}^{n}\E\int_{Q}\left|h\lambda\mathcal{O}(sh)+h\mathcal{O}_{\lambda}((sh)^{2})\right|\,|D_{i}^{2}z|^{2}\,dt.
\end{align*}
and
\begin{align*}
Y_{2}:=\,&\sum_{i=1}^{n}\E\int_{\partial_{i}Q}\left(-2s\lambda\varphi(\gamma_i)^2\partial_{i}\psi+s\mathcal{O}_{\lambda}(sh)+h\mathcal{O}_{\lambda}(sh)\right)t_{r}^{i}(|D_{i}z|^{2})\nu_{i}\,dt+\\
        &+\sum_{i=1}^{n}\E\int_{\partial_iQ}\mathcal{O}_{\lambda}(sh)t_r^{i}(|D_iz|^2)\,dt+\sum_{i,j=1}^{n}\E\int_{\partial_{i}Q}s\mathcal{O}_{\lambda}((sh)^2)\,t_{r}^{i}(|D_{i}z|^{2})\nu_{i}\,dt
\end{align*}
\end{lemma}


Then by Lemma \ref{lem:dz}-\ref{lem:dt} we obtain the following
\begin{equation}\label{eq:sumterms}
\begin{split}
    2\E &\int_Q CzBz \geq\,\E\int_{Q}2s^{3}\lambda^{4}\varphi^{3}|\nabla_{\gamma}\psi|^{4}\,|z|^{2}\,dt+ \sum_{i=1}^{n}\E\int_{Q_{i}^{\ast}}2s\lambda^{2}\varphi\gamma_i|\nabla_{\gamma}\psi|^{2}\,|D_{i}z|^{2}\,dt\\
    &-\sum_{i=1}^{n} \E\int_{Q^{\ast}}\gamma_i |D_i(dz)|^2+\E\int_{Q}(s^2\lambda^2\varphi^2+s\lambda^2\varphi)|\nabla_{\gamma} \psi|^2|dz|^2-\sum_{j=1}^{3}X_{j}-\sum_{j=1}^2 Y_j.
    \end{split}
\end{equation}
To give an estimate of the right-hand side of \eqref{eq:Estirf}, we need the following estimate $M_hz$. The proof can be adapted from Lemma 4.2 in  \cite{boyer-2010-1d-elliptic} and the estimation of $\Phi$ in \cite{zhao:2024}.
\begin{lemma}\label{lem:estimate:termsMh}
(\textit{Estimate of $M_hz$.)} For $\displaystyle \max_{t\in[0,T]}\{\theta(t)\}\tau h\leq 1$, we have
\begin{equation*}
    \E\int_{Q}|M_hz|^2\,dt\leq \mathcal{O}_{\lambda}(1)\left(\E\int_{Q}s^2|z|^2\,dt+h^2\sum_{i=1}^{n}\int_{Q^{\ast}}s^2|D_iz|^2\,dt\right).
\end{equation*}
\end{lemma}
In addition, we will need the following lemma, where its proof proceeds analogously to the proof of Lemma 3.11 in \cite{boyer-2014}.
\begin{lemma}\label{lem:inequalityforDz}
    For $\displaystyle \max_{t\in[0,T]}\{\theta(t)\}\tau h\leq 1$, we have
    \begin{equation*}
         \begin{split}
        &\sum_{i=1}^{n} \E\int_{Q_i^{\ast}} s\lambda^2\varphi \gamma_i |\nabla_{\gamma}\psi|^2|D_iz|^2\,dt\geq\frac{1}{n}\sum_{i=1}^{n}\E\int_{Q} s\lambda^2\varphi \gamma_i|\nabla_{\gamma}\psi|^2|A_iD_iz|^2\,dt+X_3+Y_3,
    \end{split}
    \end{equation*}
 where 
 \begin{equation*}
     \begin{split}
         &X_3:=\, \sum_{i\ne =1}^{n}\left(\E\int_{Q_{ij}^{\ast}}h\mathcal{O}_{\lambda}(sh)\,|D^2_{ij}z|^2\,dt-\E\int_{Q_{i}^{\ast}}h\mathcal{O}(sh)|D_iz|^2\,dt\right)\\
        &+\sum_{i=1}^{n}\left(\E\int_{Q} h\mathcal{O}_{\lambda}(sh)|A_iD_iz|^2\,dt+\E\int_{Q} h\mathcal{O}_{\lambda}(sh)|D^2_iz|^2\,dt-\E\int_{Q_i^{\ast}} h\mathcal{O}_{\lambda}(sh)|D_iz|^2\,dt\right)
     \end{split}
 \end{equation*}
 and
 \begin{equation*}
     Y_3:=\, \sum_{i=1}^{n}\E\int_{\partial_iQ} h\mathcal{O}(sh)t_r^i(|D_iz|^2)\,dt.
 \end{equation*}
\end{lemma}
Combining the lemma \ref{lem:estimate:termsMh} and \ref{lem:inequalityforDz} with \eqref{eq:sumterms}, we see that for $sh\leq \varepsilon$, there exist $\lambda_{1}\geq 1$, and $\varepsilon_{1}(\lambda)>0$ such that for $\lambda\geq \lambda_{1}$ and $0<sh\leq \varepsilon_{1}(\lambda)$, we have
\begin{equation}\label{eq:conmutador}
\begin{split}
\E\int_{Q}|rf|^2\,dt&\geq\,\E\int_{Q}|Cz|^2\,dt-\sum_{i=1}^{n}\E\int_{Q_{i}^{\ast}}\gamma_i|D_i(dz)|^2+\E\int_{Q}2s^{3}\lambda^{4}\varphi^{3}|\nabla_{\gamma}\psi|^{4}\,|z|^{2}\,dt\\
   &+ \sum_{i=1}^{n}\E\int_{Q_{i}^{\ast}}s\lambda^{2}\varphi|\nabla_{\gamma}\psi|^{2}\,|D_{i}z|^{2}\,dt+\frac{1}{n}\sum_{i=1}^{n}\E\int_{Q} s\lambda^2\varphi \gamma_i|\nabla_{\gamma}\psi|^2|A_iD_iz|^2\,dt\\
   &+\E\int_{Q}(s^2\lambda^2\varphi^2+s\lambda^2\varphi)|\nabla_{\gamma} \psi|^2|dz|^2+H+\tilde{X}+\tilde{Y},
   \end{split}
\end{equation}
with
\begin{equation}
\begin{split}
\tilde{X}:=&\sum_{l=1}^{2}X_{l}+Y_{l}-\E\int_{Q}s^{3}\lambda^{3}\varphi^{3}\mathcal{O}(1)|z|^{2}\,dt
    \end{split}-\E\int_{Q}s^2\lambda^4\varphi^2\mathcal{O}(1)\,|z|^2\,dt
\end{equation}
If we fix $\lambda=\lambda_{1}$, we can choose $\varepsilon_{0}$ and $h_{0}$ to be sufficiently small, with $0<\varepsilon \leq \varepsilon(\lambda_{1})$, and $s_{0}\geq 1$ sufficiently large, so that for $s\geq \tau_{0}$, $0<h\leq h_{0}$, and $sh\leq \varepsilon_{0}$, we obtain 
\begin{equation}\label{eq:firstestimate}
\begin{split}
\mathcal{C}_{\tau_{0},\varepsilon_0}&\E\int_{Q}|rf|^2\,dt\geq\,\E\int_{Q}|Cz|^2\,dt-\sum_{i=1}^{n}\E\int_{Q_i^{\ast}}\gamma_i|D_i(dz)|^2+\E\int_{Q}s^{3}\lambda^{4}_{1}\varphi^{3}|\nabla_{\gamma}\psi|^{4}\,|z|^{2}\,dt\\
    &+ \sum_{i=1}^{n}\E\int_{Q_{i}^{\ast}}s\lambda^{2}_{1}\varphi|\nabla_{\gamma}\psi|^{2}\,|D_{i}z|^{2}+\frac{1}{n}\sum_{i=1}^{n}\E\int_{Q} s\lambda_{1}^2\varphi \gamma_i|\nabla_{\gamma}\psi|^2|A_iD_iz|^2\,dt\\
    &+\E\int_{Q}(s^2\lambda_{1}^2\varphi^2+s\lambda_{1}^2\varphi)|\nabla_{\gamma} \psi|^2|dz|^2-\tilde{Y},
    \end{split}
\end{equation}
where
\begin{equation}\label{eq:valueboundary}
    \begin{split}
        \tilde{Y}:= \,\sum_{i=1}^{n}\E\int_{\partial_i Q }2s\lambda_{1} \varphi (\gamma_i)^2\partial_i\psi\,& t_r^i(|D_iz|^2)\nu_i\,dt\\
        &+\mathcal{C}_{\tau_{0},\varepsilon_0}\left(\left.-\sum_{i=1}^{n}\E\int_{\mathcal{M}_i^{\ast}}|D_iz|^2\right|_0^T+\left.\E\int_{\mathcal{M}}s^2\,|z|^2\right|_0^T\right)
    \end{split}
\end{equation}
{ 
Moreover, since $z(0)=0$ and recalling that $\psi$ satisfies \eqref{Assumption_data_inteior}, we obtain the following
\begin{equation*}
    \tilde{Y}\leq \mathcal{C}_{\tau_0,\varepsilon_0}\left(\E\int_{\mathcal{M}}s^2|z|\bigg|_{t=T}\right).
\end{equation*}
Therefore, we write \eqref{eq:firstestimate} as
\begin{equation}\label{eq:finalestimateinz2}
\begin{split}  
\E\int_{Q} s^{3}\lambda^{4}_{1}&\varphi^{3}|z|^{2}\,dt+\E\int_{Q}(s^2\lambda_{1}^2\varphi^2+s\lambda_{1}^2\varphi)|dz|^2+\E\int_{Q}|Cz|^2\,dt\\
&+\sum_{i=1}^{n}\left(\E\int_{Q_{i}^{\ast}}s\lambda^{2}_{1}\varphi\,|D_{i}z|^{2}\,dt+\E\int_{Q} s\lambda_{1}^2\varphi|A_iD_iz|^2\,dt-\E\int_{Q_i^{\ast}}|D_i(dz)|^2\right)\\
&\leq\,\E\int_{0}^T\int_{G_{1}\cap \mathcal{M}}s^{3}\lambda^{4}_{1}\varphi^{3}\,|z|^{2}\,dt+ \sum_{i=1}^{n}\E\int_0^T\int_{G_{1}\cap\mathcal{M}_{i}^{\ast}}s\lambda^{2}_{1}\varphi\,|D_{i}z|^{2}\,dt\\
&+\mathcal{C}_{\tau_{0},\varepsilon_0}\left(\E\int_{Q}|rf|^2\,dt+\left.\E\int_{\mathcal{M}}s^2|z|^2\right|_{t=T}\right).
\end{split}
\end{equation}
}

\subsection{Returning to the original variable}\label{sec:returninbackvarible}

We now return to our original variable. The methodology in this section departs from the approach in \cite{LPP:2025}. Therefore, as a first step, we establish the following result for the terms on the left-hand side of the desired Carleman estimate, which corresponds to the left-hand side of inequality \eqref{eq:finalestimateinz2}.

\begin{lemma}\label{lem:inequalityDz-Dw}
Let $\lambda=\lambda_{1}$, then there exists $\varepsilon_{2}$ and $h_{2}$ sufficiently small, with $0<\varepsilon \leq \varepsilon(\lambda_{1})$, and $s_{2}\geq 1$ sufficiently large, such that for $s\geq s_{2}$, $0<h\leq h_{2}$, and $sh\leq \varepsilon_{2}$, we obtain the following inequality:
\begin{equation}\label{eq:estimateleftside}
\begin{split}
    J(w)+&\E\int_Q s^3\lambda^4_{1}\varphi^3|rw|^2\,dt+\E\int_{Q}s\lambda^2_{1}\varphi|rg|^2\,dt-\sum_{i=1}^{n}\E\int_{Q_i^{\ast}}|rD_i(g)|^2\,dt\\
    &\leq \mathcal{C}'_{\tau_2,\varepsilon_2}\left(\E\int_Qs^3\lambda^4_{1}\varphi^3|z|^2\,dt+\E\int_{Q}(s^2\lambda^2\varphi^2+s\lambda^2\varphi)|\nabla_{\gamma} \psi|^2|dz|^2\right.\\
    &\left.+\sum_{i=1}^{n} \left[\E\int_{Q}s\lambda^2\varphi|A_iD_iz|^2\,dt+\E\int_{Q_i^{\ast}}s\varphi\lambda^2\,|D_iz|^2\,dt-\E\int_{Q_i^{\ast}}|D_i(dz)|^2\right]\right)
\end{split}
\end{equation}
where $\displaystyle J(w):= \sum_{i=1}^{n}\E\int_{Q_i^{\ast}}s\varphi\lambda^2|rD_iw|^2\,dt +\E\int_{Q}s\lambda^2\varphi |rA_iD_iw|^2\,dt$ and $w$ is solution of $\displaystyle dw+\sum_{i=1}^{n}D_i(\gamma_i w)dt=fdt+gdB(t)$ with $w=0$ on $\partial Q$.
\end{lemma}
\begin{proof}
Recalling $w=\rho z$, we have $D_{i}w=A_i\rho\,D_iz+D_i\rho\,A_iz$. Therefore, by \cite[Lemma 3.7]{BHLR:2010b} and {\cite[Proposition 3.3]{AA:perez:2024}}, for each $i{ =1,...,n.}$ it holds that
\begin{equation}\label{eq:equal1_1}
    \begin{split}
        \E\int_{Q_{i}^{\ast}}s\varphi \lambda^2|rD_i w|^2\,dt\leq\;& \E\int_{Q_i^{\ast}}s\varphi\lambda^2\,|D_iz|^2\,dt+\E\int_{Q_{i}^{\ast}}s\mathcal{O}_{\lambda}((sh)^2)\,|D_iz|^2\,dt\\
        &\E\int_{Q_i^{\ast}}s^3\varphi^3\lambda^4|A_iz|^2\,dt+\E\int_{Q_i^{\ast}}s^3\mathcal{O}_{\lambda}((sh)^2)\,|A_iz|^2\,dt.
    \end{split}
\end{equation}
Now, using \eqref{eq:inequalityAverange}, and by the integration by parts with respect average operator, we have
\begin{equation*}
    \begin{split}
\E\int_{Q_i^{\ast}}s^3\varphi^3\lambda^4|A_iz|^2\,dt\leq  \E\int_{Q_i^{\ast}}s^3\varphi^3\lambda^4A_i(|z|^2)\,dt=\E\int_{Q}s^3A_i(\varphi^3)\lambda^4|z|^2\,dt,
    \end{split}
\end{equation*}
where used that $z=0$ on $\partial_i Q$. Applying the approximation to function $A_i(\varphi^3)=\varphi^3+\mathcal{O}_{\lambda}((sh)^2)$ on integral above, we obtain that
\begin{equation}\label{eq:equall1_2*}
\E\int_{Q_i^{\ast}}s^3\varphi^3\lambda^4|A_iz|^2\,dt\leq\E\int_{Q}s^3\varphi^3\lambda^4|z|^2\,dt+\E\int_{Q}s\mathcal{O}_{\lambda}((sh)^2)|z|^2\,dt.
\end{equation}
Similar to the derivation above, we see that
\begin{equation}\label{eq:equal1_2}
    \E\int_{Q_i^{\ast}}s^3\mathcal{O}_{\lambda}((sh)^2)\,|A_iz|^2\,dt\leq \E\int_{Q}s^3\mathcal{O}_{\lambda}((sh)^2)\,|z|^2\,dt.
\end{equation}
Then, combining \eqref{eq:equal1_1}-\eqref{eq:equal1_2}, we obtain the following inequality for first term of the $J(w)$:
\begin{equation}\label{chagelleofDi}
\begin{split}
   \E\int_{Q_{i}^{\ast}}s\varphi \lambda^2&|rD_i w|^2\,dt\leq\; \E\int_{Q_i^{\ast}}s\varphi\lambda^2\,|D_iz|^2\,dt+\E\int_{Q_{i}^{\ast}}s\mathcal{O}_{\lambda}((sh)^2)\,|D_iz|^2\,dt\\
   &\E\int_{Q}s^3\varphi^3\lambda^4|z|^2\,dt+\E\int_{Q}s\mathcal{O}_{\lambda}((sh)^2)|z|^2\,dt+\E\int_{Q}s^3\mathcal{O}_{\lambda}((sh)^2)\,|z|^2\,dt.
\end{split}
\end{equation}
Now, for the second terms of the $J(w)$ we use \eqref{eq:difference:product} and \eqref{eq:average:product}, therefore
\begin{equation*}
    \begin{split}
        A_i(D_i(w))=&A_i(D_i(\rho)A_i(z))+A_i(A_i(\rho)D_i(z))\\
        =&A_iD_i(\rho)\,A_i^2(z)+A_i^2(\rho)\,A_iD_i(z)+\frac{h^2}{4}\left(D_i^2(\rho)\,D_iA_i(z)+D_iA_i(\rho)\,D_i^2(z)\right).
    \end{split}
\end{equation*}
Applying \eqref{eq:averengeanddifference} to the first term on the right-hand side of the previous equation, we obtain the following equality 
\begin{equation}\label{eq1:lemmaAD} 
    rA_iD_iw=rA_iD_i(\rho)\, z+rA_i^2(\rho)\,A_iD_i(z)+\frac{h^2}{4}\left(rD_i^2(\rho)\,D_iA_i(z)+2rD_iA_i(\rho)\,D_i^2(z)\right).
\end{equation}
With Proposition {\cite[Proposition 3.3]{AA:perez:2024}} and \cite[Lemma 3.7]{BHLR:2010b} we then find
\begin{equation}\label{eq2:lemAdz}
    rA_iD_iw=A_iD_iz-s\lambda\varphi\mathcal{O}(1)\,z+\mathcal{O}_{\lambda}((sh)^2)\,D_iA_iz+h\mathcal{O}_{\lambda}(sh)\, D_i^2(z)
\end{equation}
Now, by applying the triangle inequality and \eqref{eq2:lemAdz},   we obtain the following estimate for each $i{ =1,...,n.}$ of the second terms from $J(w)$:  
\begin{equation}\label{eq:changelleAiDi}
    \begin{split}
      \E\int_{Q}s\lambda^2\varphi &|rA_iD_iw|^2\,dt\\
      \leq& \, \E\int_{Q}s\lambda^2\varphi|A_iD_iz|^2\,dt+\E\int_{Q}s^3\lambda^4\varphi\mathcal{O}(1)\,|z|^2\,dt+\E\int_{Q}\mathcal{O}_{\lambda}((sh)^4)\,|D^2_iz|^2\,dt.
    \end{split}
\end{equation}
Combining \eqref{chagelleofDi} and \eqref{eq:changelleAiDi}, we can estimate to $J(w)$ as 
\begin{equation}\label{eq:estimateofJ(w)}
\begin{split}
    J(w)\leq \sum_{i=1}^{n}&\left[\E\int_{Q_i^{\ast}}s\varphi\lambda^2\,|D_iw|^2\,dt+\E\int_{Q}s\lambda^2\varphi|A_iD_iz|^2\,dt+\E\int_{Q_{i}^{\ast}}s\mathcal{O}_{\lambda}((sh)^2)\,|D_iz|^2\,dt\right.\\
    &\left.+\E\int_{Q}\mathcal{O}_{\lambda}((sh)^4)\,|D^2_iz|^2\,dt\right]+\E\int_{Q}s^3\lambda^4\varphi\mathcal{O}(1)\,|z|^2\,dt.
\end{split}
\end{equation}
Now, we will focus on estimating the two terms on the left-hand side of the inequality in this Lemma. Given the earlier definition $w=\rho z$ we can derive the relationship between their differentials: $dw=\rho\,dz+\partial_{t}(\rho)\,zdt$. Consequently, applying Proposition \ref{pro:product} we obtain   $$rA_i(\rho)D_i(dz)=rD_i(dw)-rD_i(\partial_t(\rho))A_i(z)dt-rA_i(\partial_t(\rho))D_i(z)dt-rD_i(\rho)A_i(dz).$$ 
Now, by {\cite[Lemma 3.7]{BHLR:2010b}}, {\cite[Corollary 3.8]{BHLR:2010b}} and recall that $r\rho=1$ it follows that
\begin{align*}
D_i(dz)=&rD_i(dw)-\mathcal{O}_{\lambda}((sh)^2)D_i(dz)+s^2\lambda^2\varphi ^2\theta_t(\partial_i\psi)A_i(z)dt-s\lambda^2\varphi\theta_t(\partial_i\psi)A_i(z)dt\\
&-s\lambda\varphi \theta_t D_i(z)dt+s\varphi\lambda(\partial_i\psi)A_i(dz)+\mathcal{O}_\lambda((sh)^2)D_iz\,dt-s\mathcal{O}_{\lambda}((sh)^2)A_i(dz)
\end{align*}
Therefore, from the above equation for each $i{ =1,...,n.}$, we obtain that
\begin{equation*}
    \begin{split}
        \E\int_{Q_i^{\ast}}\gamma_i|D_i(dz)|^2=& \E\int_{Q_i^{\ast}}\gamma_i|rD_i(w)|^2+\E\int_{Q_i^{\ast}}\mathcal{O}_{\lambda}((sh)^4)|D_i(dz)|^2\\
        &+\E\int_{Q_i^{\ast}}s^2\lambda^2\gamma_i(\partial_i\psi)^2|A_i(dz)|^2+\E\int_{Q_i^{\ast}}s^2\mathcal{O}_{\lambda}((sh)^4)|A_i(dz)|^2.
    \end{split}
\end{equation*}
Similar to \eqref{eq:equall1_2*} and \eqref{eq:equal1_2}, we find that
\begin{equation}\label{eq:Di(dz)}
\begin{split}
\E\int_{Q_i^{\ast}}\gamma_i|D_i(dz)|^2\leq& \E\int_{Q_i^{\ast}}\gamma_i|rD_i(w)|^2+\E\int_{Q}s^2\lambda^2\varphi^2(\partial_i\psi)^2|dz|^2\\
&+\E\int_{Q_i^{\ast}}\mathcal{O}_{\lambda}((sh)^4)|D_i(dz)|^2+\E\int_{Q}\mathcal{O}_{\lambda}(sh)^4)|dz|^2.
\end{split}
\end{equation}
Noting that 
\begin{equation*}
    \E\int_{Q}s\lambda^2\varphi\gamma_i(\partial_i\psi)^2|dz|^2=\E\int_{Q}s\lambda^2\varphi\gamma_i(\partial_i\psi)^2|rdw|^2
\end{equation*}
Then, from \eqref{eq:Di(dz)} we obtain for each $i{ =1,...,n.}$
\begin{equation}\label{eq:estimatefinaldedw}
\begin{split}
&\E\int_{Q}s\lambda^2\varphi\gamma_i(\partial_i\psi)^2|rdw|^2-\E\int_{Q_i^{\ast}}\gamma_i|rD_i(w)|^2\leq -\E\int_{Q_i^{\ast}}\gamma_i|D_i(dz)|^2\\
&+\E\int_{Q_i^{\ast}}\mathcal{O}_{\lambda}((sh)^4)|D_i(dz)|^2+\E\int_{Q}\left(s^2\lambda^2\varphi^2+s\lambda^2\varphi\right)\gamma_i(\partial_i\psi)^2|dz|^2+\E\int_{Q}\mathcal{O}_{\lambda}(sh)^4)|dz|^2.
\end{split}
\end{equation}
Du to the inequalities \eqref{eq:estimateofJ(w)},\eqref{eq:estimateD2} and \eqref{eq:estimatefinaldedw}, we conclude the proof of Lemma \eqref{lem:inequalityDz-Dw}.This shows that for $0<sh\leq \varepsilon$ and $\lambda=\lambda_1$, there exist $\varepsilon_2(\lambda)>0$ and $s_2\geq1$ such that for $s\geq s_2$, $0<h<h_2,$ and $0<sh\leq\varepsilon_2(\lambda)$, the desired inequality is obtained.
\end{proof}

To proceed, we introduce the following lemma, which allows us to incorporate the second-order spatial term $D_{ij}^2$ into the left-hand side of the Carleman estimate. The proof follows by adapting the approach in \cite[Theorem 1.2]{LLRP:2025}.

\begin{lemma}\label{lem:inequalityDDz} 
There exists $\varepsilon_3(\lambda)>0$ and $\tau_3\geq1$ such that for $s\geq s_3$, $0<h<h_3$, and $0<sh\leq \varepsilon_3(\lambda)$, we obtain the following inequality:
    \begin{equation*}
    \begin{split}
        &\sum_{i,j=1}^{n}\E\int_{Q_{ij}^{\ast}}s^{-1}e^{2s\varphi}|D_{ij}^2w|^2\, dt-\sum_{i=1}^{n}\E\int_{Q_i}s\lambda^2\varphi^2|\nabla\psi|_{\gamma}^2\gamma_i \partial_i\psi\,|rD_iw|^2+\mathcal{O}(s^{-1})|rA_iD_iw|^2\,dt\\
        &\leq \mathcal{C}_{\varepsilon_3,\tau_3}\left(\E\int_{Q}|Cz|^2\,dt+\E\int_{Q}s\lambda^4\varphi^2|\nabla_{\gamma}\psi|^2|z|^2\,dt+\sum_{i=1}^{n}\E\int_{Q}s\lambda^2\varphi^2|\nabla_{\gamma}\psi|^2\,|A_iD_iz|^2\,dt\right)
    \end{split}
    \end{equation*}
\end{lemma}
\begin{proof}
    Recalling that $\displaystyle r{ \sum_{i=1}^{n}}D_i(\gamma_iD_iw)=Cz+B_2z+B_3z-M_hz:= \mathcal{B}w$ and applying Young's inequality, we obtain
\begin{equation}
    |\mathcal{B}w|^2\leq  2(|Cz|^2+|B_2z|^2+|B_3z|^2+|M_hz|^2)
\end{equation}
Therefore, taking $s>1$ we have a first estimate for the third left-side term of \eqref{eq:estimateleftside} given by
\begin{equation}\label{eqfirsttermthird}
\begin{split}
       \E\int_{Q}s^{-1}|\mathcal{B}w|^2\,dt\leq \E\int_{Q}2|Cz|^2\,dt+&\E\int_{Q}2s^{-1}|B_2z|^2\,dt\\
       &+\E\int_{Q}2s^{-1}|B_3z|^2\,dt+\E\int_{Q}2s^{-1}|M_hz|^2\,dt. 
\end{split} 
\end{equation}
From the definition of the $B_2z$ and $B_3z$ respectively, \cite[Lemma 3.7]{BHLR:2010b} and Proposition {\cite[Proposition 3.3]{AA:perez:2024}}, we can see that
\begin{equation}
\begin{split}
    \E\int_{Q}2s^{-1}|B_2z|^2dt\leq&\sum_{i=1}^{n}\E\int_{Q}8s^{-1}|\gamma_irD_iA_i\rho\, D_iA_iz|^2dt\\
    \leq &8\sum_{i=1}^{n}\E\int_{Q}(s\lambda^2\varphi^2 |\nabla_\gamma\psi|^2+\mathcal{O}_{\lambda}|((sh)^4)||A_iD_iz|^2dt,
\end{split} 
\end{equation}
\begin{equation}
    \E\int_{Q}2s^{-1}|B_3z|^2dt=4\E\int_{Q}2s^{-1}|s\Delta_{\gamma_i}\varphi\, z|^2dt\leq 8\E\int_{Q}(s\lambda^4\varphi^2 |\nabla_\gamma\psi|^4+s\lambda|\mathcal{O}((1)|^2|z|^2\;dt
\end{equation}
and similar to Lemma \ref{lem:estimate:termsMh}, we can deduce the following estimate for $M_hz$ as
\begin{equation} \label{eq:Mhzi}
     \E\int_{Q}2s^{-1}|M_hz|^2\,dt\leq \mathcal{O}_{\lambda}(1)\left(\E\int_{Q}s|z|^2\,dt+h^2\sum_{i=1}^{n}\int_{Q^{\ast}}s|D_iz|^2\,dt\right).
\end{equation}
Thus, combining \eqref{eqfirsttermthird}-\eqref{eq:Mhzi}  we obtain the estimate 
\begin{equation}\label{eq:estimateD2}
    \begin{split}
        \E\int_{Q}&s^{-1}|\mathcal{B}w|^2\,dt\\
        \leq&\; \E\int_{Q}|Cz|^2dt+\E\int_{Q}(4s\lambda^4\varphi^2 |\nabla_\gamma\psi|^4+s\mathcal{O}_{\lambda}(1)+s\lambda|\mathcal{O}((1)|^2|z|^2dt\\
        &+\sum_{i=1}^{n}\left(\E\int_{Q}h\mathcal{O}_{\lambda}(sh)|D_iz|^2dt+\E\int_{Q}(4s\lambda^2\varphi ^2|\nabla_\gamma\psi|^2+\mathcal{O}_{\lambda}|((sh)^4)||A_iD_iz|^2dt\right).
    \end{split}
\end{equation}
In turn, our next task is to compare the terms $\displaystyle \E\int_{Q}s^{-1}|\mathcal{B}w|^2\, dt$ and \\ $\displaystyle \sum_{i,j=1}^{n}\E\int_{Q_{ij}^{*}}s^{-1}\gamma_i\gamma_jr^2|D_{ij}w|^2\; dt$. To this end, we can now proceed analogously to the proof of Theorem 1.2 in \cite{LLRP:2025}. In particular, from the inequalities (2.27)-(2.32) in \cite{LLRP:2025} it follows that
\begin{equation}\label{eq:inequalitiAF}
    \begin{split}
        \E\int_{Q}s^{-1}|\mathcal{B}w|^2\,dt
        \geq & { \sum_{i,j=1}^{n}}\int_{Q_{ij}^{\ast}}s^{-1}\gamma_i\gamma_j|rD_{ij}^2w|^2\,dt-\sum_{i=1}^{n}\E\int_{Q_i}s\lambda^2\varphi^2|\nabla\psi|_{\gamma}^2\gamma_i\partial_i\psi|rD_iw|^2\,dt\\
    &-\sum_{i=1}^{n}\E \int_{Q}s^{-1}r^2|\mathcal{O}(1)|^2|A_iD_iw|^2\,dt-K(w),
    \end{split}
\end{equation}
where $K(w)$ is given by
\begin{align*}
    K(w):=&{ \sum_{i,j=1}^{n}}\E\int_{Q_{ij}^{\ast}}s^{-1/2}\left(\mathcal{O}(h)+\mathcal{O}_{\lambda}((sh)^2+\mathcal{O}_{\lambda}(1)\right)|rD_{ij}^2w|^2\,dt\\
    &+{ \sum_{j=1}^{n}}\E\int_{Q_j^{\ast}}\left(s^{-1/2}\left(\mathcal{O}(h)+\mathcal{O}_{\lambda}((sh)^2+\mathcal{O}_{\lambda}(1)\right)+\mathcal{O}_{\lambda}(1)+s\mathcal{O}(sh)\right)|rD_jw|^2\,dt.
\end{align*}
Finally, there exist $\varepsilon_3(\lambda)>0$ and $s_3\geq1$ such that combining \eqref{eq:inequalitiAF} and \eqref{eq:estimateD2} for $s\geq s_3$, $0<h<h_3$, and $0<sh\leq \varepsilon_3(\lambda)$, the desired inequality is obtained
\end{proof}
The two previous lemmas allow us to obtain a lower bound for the inequality \eqref{eq:finalestimateinz2} in terms of $w$. Now, we need to obtain an upper bound for the same inequality. To achieve this, we only need to estimate the second term on the right-hand side of the inequality \eqref{eq:finalestimateinz2}, since the other terms are immediately available in terms of $w$. 
\\
{ 
From \cite[Lemma 2.10]{LPP:2025}, we first obtain that
\begin{equation*}
        \begin{split}
        \E\int_0^T&\int_{G_{1}\cap\mathcal{M}_i^{\ast}}s\varphi\lambda^2|D_iz|^2\,dt\\
        &\leq\, \mathcal{C}\left(\E\int_0^T\int_{G_{1}\cap\mathcal{M}_i^{\ast}}s\varphi\lambda^2\,|rD_iw|^2\,dt+\E\int_0^T\int_{G_{1}\cap \mathcal{M}_{i}^{\ast}}s\mathcal{O}_{\lambda}((sh)^2)\,|D_iz|^2\,dt\right.\\
        &\left. \E\int_0^T\int_{G_{1}\cap\mathcal{M}}s^3\varphi^3\lambda^4|z|^2\,dt+\E\int_0^T\int_{G_{1}\cap\mathcal{M}}s^3\mathcal{O}_{\lambda}((sh)^2)\,|z|^2\,dt\right)
    \end{split}
\end{equation*}
for $sh\leq \varepsilon_0$ and $i=1,..,n$. Moreover, taking into account that $\theta(t)$ satisfies \eqref{theta-delta}, $\left.z\right|_{t=0}=0$ and using \cite[Lemma 2.12]{LPP:2025}, we have
\begin{equation*}
    \begin{split}
            \sum_{i=1}^{n}\E\int_0^T\int_{G_{1}\cap\mathcal{M}_{i}^{\ast}}s\lambda^{2}_{1}\varphi e^{2s\varphi}\,|D_{i}w|^{2}\,dt\leq \sum_{i=1}^{n}&\E\int_0^T\int_{G_{0}\cap\mathcal{M}_{i}^{\ast}}s\lambda^{2}_{1}\varphi e^{2s\varphi}\,|D_{i}w|^{2}\,dt\\
            \leq \mathcal{C}_{\tau_3,\varepsilon_0}\left(\E\int_0^T\int_{G_{0}\cap\mathcal{M}}s^3\lambda^2\varphi ^3e^{2s\varphi}\,|w|^2\,dt+\E\int_0^T\right.&\int_{G_{0}\cap\mathcal{M}}\lambda^{-2}e^{2s\varphi}|f|^2\\
            &\left.+\left.\E\int_{G_{0}\cap\mathcal{M}}s\varphi e^{2s\varphi}\,|w|^2\right|_{t=T}\right).
        \end{split}
\end{equation*}
Thus, combining the above two inequality, we can conclude that
\begin{equation}
    \begin{split}
        \E\int_0^T&\int_{G_{1}\cap\mathcal{M}_i^{\ast}}s\varphi\lambda^2|D_iz|^2\,dt \leq \mathcal{C}_{\tau_3,\varepsilon_3}\left(\E\int_0^T\int_{G_{0}\cap\mathcal{M}}s^3\lambda^2\varphi ^3e^{2s\varphi}\,|w|^2\,dt\right.\\
            &\left.+\E\int_0^T\int_{G_{0}\cap\mathcal{M}}\lambda^{-2}e^{2s\varphi}|f|^2+\left.\E\int_{G_{0}\cap\mathcal{M}}s\varphi e^{2s(T)\varphi}\,|w|^2\right|_{t=T}\right)
    \end{split}
\end{equation}
for $sh<\varepsilon_0$ and $i=1,\ldots,n$.} Consequently, applying the inequality mentioned above, \eqref{eq:finalestimateinz2}, Lemma \ref{lem:inequalityDz-Dw} and Lemma \ref{lem:inequalityDDz}, we can determine that for $s_0=\max\{s_1,s_2,s_3\}$ and $h_0=\min\{h_1,h_2,h_3\}$, the desired Carleman estimate is obtained.
\section{Technical steps for the cross-product estimate}\label{sec:proof:carleman}

This section derives additional estimates not covered in \cite{LPP:2025}, which are essential to establish lemmas \ref{lem:dz} and \ref{lem:AdditionalsTerms}.

We begin by estimating $I_{61}$. Let us define $C_6 z = -2 s \Delta_{\gamma}(\varphi) z$ and $B_1 z = dz$. Then
\begin{equation*}
2\E\int_{Q}C_6zB_1z=-2\E\int_{Q}2s\Delta_{\gamma}(\varphi)zdz:=\, I_{61}.
\end{equation*}
Applying Itô's formula yields the following result
\begin{equation*}
I_{61}=-2\left.\E\int_{Q}s\Delta_{\gamma}(\varphi)|z|^2\right|_0^T+2\E\int_{Q}\partial_{t}(s)\Delta_{\gamma}(\varphi)|z|^2\,dt+\E\int_{Q}s\Delta(\varphi)|dz|^2.
\end{equation*}
Observing that $\Delta_{\gamma}(\varphi) = \lambda^2 \varphi |\nabla_{\gamma} \psi|^2 + \lambda \varphi \mathcal{O}(1)$ and $|\theta_t| \leq \lambda \beta T \theta$, we deduce
\begin{equation*}
I_{61}\geq\,\E\int_{Q}s\lambda^2\varphi|\nabla_{\gamma}\psi|^2|dz|^2+X_{61}
\end{equation*}
with
\begin{equation*}
\begin{aligned}
X_{61}:=\,&-2\left.\E\int_{Q}s(\lambda^2\varphi|\nabla_{\gamma}\psi|^2+\lambda \varphi \mathcal{O}(1))|z|^2\right|_0^T 
+\E\int_{Q}s\lambda\varphi\mathcal{O}(1)|dz|^2\\
&-2\E\int_{Q}s^2\lambda\beta T(\lambda^2\varphi|\nabla_{\gamma}\psi|^2+\lambda \varphi \mathcal{O}(1))|z|^2\,dt.
\end{aligned}
\end{equation*}

We now proceed to estimate $I_{62}$. For this purpose, we set $C_6 z = -2 s \Delta_{\gamma}(\varphi) z$, $\displaystyle B_2 z \, dt = -2 \sum_{i=1}^{n} \gamma_i r D_i A_i \rho \, D_i A_i z \, dt.$ 
Defining $\beta_{62}^i = 4 s \gamma_i \Delta_{\gamma}(\varphi) r D_i A_i \rho$ and applying the discrete integration by parts formula \eqref{eq:int:dif}, we obtain the following
\begin{align*} 
\E \int_Q \beta_{62}^i \, z \, D_i(A_i z) \, dt 
&= \E \int_{Q_i^\ast} D_i (\beta_{62}^i z) A_i z \, dt + \int_{\partial_i Q} \beta_{62}^i t_r^i(A_i z) z \, \nu_i \, dt \\
&= \E \int_{Q_i^\ast} D_i (\beta_{62}^i) |A_i z|^2 \, dt - \frac{1}{2} \E \int_{Q_i^\ast} A_i (\beta_{62}^i) D_i(|z|^2) \, dt,
\end{align*}
where we used $z = 0$ on $\partial_i Q$ and the product rule \eqref{eq:difference:product}.

Next, applying the inequality \eqref{eq:inequalityAverange} and the integration by parts formula for the average operator \eqref{eq:int:ave}, we arrive at
\begin{align*}
\E \int_Q \beta_{62}^i \, z \, D_i(A_i z) \, dt 
&\geq \E \int_{Q_i^\ast} D_i(\beta_{62}^i) A(|z|^2) \, dt + \frac{1}{2} \int_Q D_i A_i (\beta_{62}^i) |z|^2 \, dt \\
&\geq - \frac{1}{2} \E \int_Q A_i D_i (\beta_{62}^i) |z|^2 \, dt - \frac{h}{2} \int_{\partial_i Q} |z|^2 \, t_r^i(D_i(\beta_{62}^i)) \, dt.
\end{align*}
Using the results of \cite{AA:perez:2024} and the boundary condition $z = 0$ on $\partial_i Q$, we finally deduce
\[
I_{62} \geq 2 \, \E \int_Q \big( s^2 \lambda^3 \varphi^2 + s^2 \lambda^4 \varphi^2 \big) |\nabla_{\gamma_i} \psi|^4 |z|^2 \, dt + \E \int_Q s \, \mathcal{O}_\lambda(1) (sh)^2.
\]
Finally, we focus on the estimate of $I_{63}$. Recall that $B_3 z = - C_6 z = -2 s \Delta_{\gamma}(\varphi) z.$ Using the expression $\Delta_{\gamma}(\varphi) = \lambda^2 \varphi |\nabla_{\gamma} \psi|^2 + \mathcal{O}_{\lambda}(1)$ and the triangle inequality, we obtain the following.
\begin{equation*}
I_{63} = -4 \, \E \int_Q s^2 |\Delta_{\gamma} \varphi|^2 \, |z|^2 \, dt 
\geq -4 \, \E \int_Q s^2 \lambda^4 \varphi^2 |\nabla_{\gamma} \psi|^4 \, |z|^2 \, dt 
- 4 \, \E \int_Q s^2|\mathcal{O}_{\lambda}(1)|^2 \, |z|^2 \, dt.
\end{equation*}
{ 
\section{Technical steps for return to the original variable for the case boundary data in the right-hand side}\label{appendix:additionalreturn}
Recalling that $w=\rho w$ and by Proposition \ref{pro:product} and \cite[Lemma 3.7]{BHLR:2010b}, we have $rD_iw=D_iz+\mathcal{O}_{\lambda}((sh)^2)D_iz-\lambda s\varphi(\partial_i\psi)A_i(z)+s\mathcal{O}_{\lambda}((sh)^2)A_i(z) $. We obtain the following
\begin{align*}
     \sum_{i=1}^{n}\E\int_0^T\int_{\Gamma}&s\lambda_{1}\varphi t_{r}^{i}(|D_iz|^2)\,dt\leq\sum_{i=1}^{n}\left(\E\int_0^T\int_{\Gamma_{+}}s\lambda_{1}\varphi t_{r}^{i}(|rD_i(w)|^2)dt\right.\\
     &\left.+\E\int_0^T\int_{\Gamma_+} s\mathcal{O}_{\lambda}((sh)^2)t_r^i(D_iz)dt+\E\int_0^T\int_{\Gamma_+}\lambda^3 s^3\varphi^3(\partial_i\psi)^2t_r^i(|A_i(z)|^2)dt\right.\\
     &\left.+\E\int_0^T\int_{\Gamma_+}s^3|\mathcal{O}_{\lambda}((sh)^2)|^2 t_r^i(|A_i(z)|^2)dt\right).
\end{align*}
Now, as $z=0$ on $\partial Q$ we notice that $t_r^i(|A_iz|^2)=\frac{h^2}{2}t_r^i(|D_iz|^2)$. Then, we can see that
\begin{align*}
     \sum_{i=1}^{n}\E\int_0^T&\int_{\Gamma}s\lambda_{1}\varphi t_{r}^{i}(|D_iz|^2)\,dt\leq\sum_{i=1}^{n}\left(\E\int_0^T\int_{\Gamma_{+}}s\lambda_{1}\varphi t_{r}^{i}(|rD_i(w)|^2)dt\right.\\
     &\left.+\E\int_0^T\int_{\Gamma_+} s\mathcal{O}_{\lambda}((sh)^2)t_r^i(D_iz)dt+\E\int_0^T\int_{\Gamma_+}s|\mathcal{O}_{\lambda}((sh)^6)| t_r^i(|D_i(z)|^2)dt\right).
\end{align*}
Therefore, by the previous inequality we can find $\varepsilon_3(\lambda)>0$ and $s_3>1$ such that for $s\geq s_3$, $0<h<h_3$ and $0<sh\leq \varepsilon_3(\lambda)$, we have
\begin{align*}
    \E\int_{Q}&|rf|^2\,dt+\sum_{i=1}^{n}\E\int_0^T\int_{\Gamma}s\lambda_{1}t_{r}^{i}(|D_iz|^2)\,dt+\left.\E\int_{\mathcal{M}}s^2|z|^2\right|_{t=T}\\
    \leq&\mathcal{C}_{\tau_3,\varepsilon_3}\left( \E\int_{Q}|rf|^2\,dt+\sum_{i=1}^{n}\E\int_0^T\int_{\Gamma_{+}}s\lambda_{1}\varphi t_{r}^{i}(|rD_i(w)|^2)dt+\left.\E\int_{\mathcal{M}}s^2|rw|^2\right|_{t=T}\right).
\end{align*}
}
\bibliographystyle{abbrv}
\bibliography{references}

@article{WZZW:2025,
  author        = {Wu, Bin and Zhu, Xu and Zhou, Wenwen and Wang, Zewen},
  title         = {Lipschitz stability for an inverse problem of a fully-discrete stochastic hyperbolic equation},
  journal       = {arXiv preprint arXiv:2512.07072},
  year          = {2025},
  eprint        = {2512.07072},
  archivePrefix = {arXiv},
  url           = {https://arxiv.org/abs/2512.07072}
}

@article {Zho92,
    AUTHOR = {Zhou, Xun Yu},
     TITLE = {A duality analysis on stochastic partial differential
              equations},
   JOURNAL = {J. Funct. Anal.},
  FJOURNAL = {Journal of Functional Analysis},
    VOLUME = {103},
      YEAR = {1992},
    NUMBER = {2},
     PAGES = {275--293},
      ISSN = {0022-1236},
   MRCLASS = {60H15 (34F05 35R60 47N20 47N30)},
  MRNUMBER = {1151549},
       DOI = {10.1016/0022-1236(92)90122-Y},
       URL = {https://doi.org/10.1016/0022-1236(92)90122-Y},
}

@article {KR77,
    AUTHOR = {Krylov, N. V. and Rozovski\u{\i}, B. L.},
     TITLE = {The {C}auchy problem for linear stochastic partial
              differential equations},
   JOURNAL = {Izv. Akad. Nauk SSSR Ser. Mat.},
  FJOURNAL = {Izvestiya Akademii Nauk SSSR. Seriya Matematicheskaya},
    VOLUME = {41},
      YEAR = {1977},
    NUMBER = {6},
     PAGES = {1329--1347, 1448},
      ISSN = {0373-2436},
   MRCLASS = {60H15 (60G35)},
  MRNUMBER = {0501350},
MRREVIEWER = {O. A. Glonti},
}

@article{AA:perez:2024,
      title={Asymptotic behavior of {C}arleman weight functions}, 
    journal = {\href{https://arxiv.org/abs/2412.19892}{arXiv:2412.19892}},
      author={Ariel A. P\'erez},
      year={2024},
    eprint={2412.19892},
      archivePrefix={arXiv},
      primaryClass={},
      url={https://arxiv.org/abs/2403.19413}, 
}

@article {zhao:2024,
    AUTHOR = {Zhao, Qingmei},
     TITLE = {Null controllability for stochastic semidiscrete parabolic
              equations},
   JOURNAL = {SIAM J. Control Optim.},
  FJOURNAL = {SIAM Journal on Control and Optimization},
    VOLUME = {63},
      YEAR = {2025},
    NUMBER = {3},
     PAGES = {2007--2028},
      ISSN = {0363-0129,1095-7138},
   MRCLASS = {93B05 (60H15)},
  MRNUMBER = {4917058},
       DOI = {10.1137/24M1639166},
       URL = {https://doi.org/10.1137/24M1639166},
}

@book {klebaner2012introduction,
    AUTHOR = {Klebaner, Fima C.},
     TITLE = {Introduction to stochastic calculus with applications},
   EDITION = {Third},
 PUBLISHER = {Imperial College Press, London},
      YEAR = {2012},
     PAGES = {xiv+438},
      ISBN = {978-1-84816-832-9; 1-84816-832-2},
   MRCLASS = {60-01 (60H05 60H30 60J65 60J70 60J75 60J85)},
  MRNUMBER = {2933773},
       DOI = {10.1142/p821},
       URL = {https://doi.org/10.1142/p821},
}

@article {LOPD:2023,
    AUTHOR = {Lecaros, Rodrigo and Ortega, Jaime H. and P\'{e}rez, Ariel and
              De Teresa, Luz},
     TITLE = {Discrete {C}alder\'{o}n problem with partial data},
   JOURNAL = {Inverse Problems},
  FJOURNAL = {Inverse Problems. An International Journal on the Theory and
              Practice of Inverse Problems, Inverse Methods and Computerized
              Inversion of Data},
    VOLUME = {39},
      YEAR = {2023},
    NUMBER = {3},
     PAGES = {Paper No. 035001, 28},
      ISSN = {0266-5611,1361-6420},
   MRCLASS = {42B37 (35R30 42B10 65M06)},
  MRNUMBER = {4541733},
       DOI = {10.1088/1361-6420/acb0f8},
       URL = {https://doi.org/10.1088/1361-6420/acb0f8},
}

@article {boyer-2010-1d-elliptic,
    AUTHOR = {Boyer, Franck and Hubert, Florence and Le Rousseau, J\'{e}r\^{o}me},
     TITLE = {Discrete {C}arleman estimates for elliptic operators and
              uniform controllability of semi-discretized parabolic
              equations},
   JOURNAL = {J. Math. Pures Appl. (9)},
  FJOURNAL = {Journal de Math\'{e}matiques Pures et Appliqu\'{e}es. Neuvi\`eme S\'{e}rie},
    VOLUME = {93},
      YEAR = {2010},
    NUMBER = {3},
     PAGES = {240--276},
      ISSN = {0021-7824},
   MRCLASS = {93B05 (35K10 65M06 93C20)},
  MRNUMBER = {2601332},
       DOI = {10.1016/j.matpur.2009.11.003},
       URL = {https://doi-org.uchile.idm.oclc.org/10.1016/j.matpur.2009.11.003},
}

@article {boyer-2014,
    AUTHOR = {Boyer, Franck and Le Rousseau, J\'{e}r\^{o}me},
     TITLE = {Carleman estimates for semi-discrete parabolic operators and
              application to the controllability of semi-linear
              semi-discrete parabolic equations},
   JOURNAL = {Ann. Inst. H. Poincar\'{e} Anal. Non Lin\'{e}aire},
  FJOURNAL = {Annales de l'Institut Henri Poincar\'{e}. Analyse Non Lin\'{e}aire},
    VOLUME = {31},
      YEAR = {2014},
    NUMBER = {5},
     PAGES = {1035--1078},
      ISSN = {0294-1449},
   MRCLASS = {35K58 (35B45 93B05)},
  MRNUMBER = {3258365},
MRREVIEWER = {Joseph L. Shomberg},
       DOI = {10.1016/j.anihpc.2013.07.011},
       URL = {https://doi-org.uchile.idm.oclc.org/10.1016/j.anihpc.2013.07.011},
}

@article {BHLR:2010b,
    AUTHOR = {Boyer, Franck and Hubert, Florence and Le Rousseau, J\'{e}r\^{o}me},
     TITLE = {Discrete {C}arleman estimates for elliptic operators in
              arbitrary dimension and applications},
   JOURNAL = {SIAM J. Control Optim.},
  FJOURNAL = {SIAM Journal on Control and Optimization},
    VOLUME = {48},
      YEAR = {2010},
    NUMBER = {8},
     PAGES = {5357--5397},
      ISSN = {0363-0129},
   MRCLASS = {35K20 (35P15 65M06 93B05 93B07)},
  MRNUMBER = {2745778},
       DOI = {10.1137/100784278},
       URL = {https://doi-org.uchile.idm.oclc.org/10.1137/100784278},
}

@article {EDG:2011,
    AUTHOR = {Ervedoza, S. and de Gournay, F.},
     TITLE = {Uniform stability estimates for the discrete {C}alder\'on
              problems},
   JOURNAL = {Inverse Problems},
  FJOURNAL = {Inverse Problems. An International Journal on the Theory and
              Practice of Inverse Problems, Inverse Methods and Computerized
              Inversion of Data},
    VOLUME = {27},
      YEAR = {2011},
    NUMBER = {12},
     PAGES = {125012, 37},
      ISSN = {0266-5611,1361-6420},
   MRCLASS = {78A46 (35B35 35J25 65N06 78A70)},
  MRNUMBER = {2854330},
MRREVIEWER = {Yongzhi\ Xu},
       DOI = {10.1088/0266-5611/27/12/125012},
       URL = {},
}

@book {fursikov-1996,
    AUTHOR = {Fursikov, A. V. and Imanuvilov, O. Yu.},
     TITLE = {Controllability of evolution equations},
    SERIES = {Lecture Notes Series},
    VOLUME = {34},
 PUBLISHER = {Seoul National University, Research Institute of Mathematics,
              Global Analysis Research Center, Seoul},
      YEAR = {1996},
     PAGES = {iv+163},
   MRCLASS = {93-02 (35B37 35Q30 93B05 93C20)},
  MRNUMBER = {1406566},
MRREVIEWER = {Vilmos Komornik},
}

@article{Yamamoto_2009,
doi = {10.1088/0266-5611/25/12/123013},
url = {https://dx.doi.org/10.1088/0266-5611/25/12/123013},
year = {2009},
month = {dec},
publisher = {},
volume = {25},
number = {12},
pages = {123013},
author = {Yamamoto, Masahiro},
title = {Carleman estimates for parabolic equations and applications},
journal = {Inverse Problems}
}

@article{QLu:2012,
doi = {10.1088/0266-5611/28/4/045008},
url = {https://dx.doi.org/10.1088/0266-5611/28/4/045008},
year = {2012},
month = {03},
publisher = {IOP Publishing},
volume = {28},
number = {4},
pages = {045008},
author = {Lü, Qi},
title = {Carleman estimate for stochastic parabolic equations and inverse stochastic parabolic problems},
journal = {Inverse Problems}
}

@article{ref37,
url = {https://doi.org/10.1515/jiip-2017-0003},
title = {Inverse problems for stochastic parabolic equations with additive noise},
author = {Ganghua Yuan},
pages = {93--108},
volume = {29},
number = {1},
journal = {Journal of Inverse and Ill-posed Problems},
doi = {doi:10.1515/jiip-2017-0003},
year = {2021},
lastchecked = {2025-02-12}
}

@article{ref1,
doi = {10.1088/0266-5611/26/7/074014},
url = {https://dx.doi.org/10.1088/0266-5611/26/7/074014},
year = {2010},
month = {07},
publisher = {},
volume = {26},
number = {7},
pages = {074014},
author = {Bao, Gang and Chow, Shui-Nee and Li, Peijun and Zhou, Haomin},
title = {Numerical solution of an inverse medium scattering problem with a stochastic source},
journal = {Inverse Problems}
}

@article{ref2,
doi = {10.1088/0266-5611/29/1/015006},
url = {https://dx.doi.org/10.1088/0266-5611/29/1/015006},
year = {2012},
month = {12},
publisher = {IOP Publishing},
volume = {29},
number = {1},
pages = {015006},
author = {Bao, Gang and Xu, Xiang},
title = {An inverse random source problem in quantifying the elastic modulus of nanomaterials},
journal = {Inverse Problems}
}

@misc{ref13,
      title={Stability and regularization for ill-posed {C}auchy problem of a stochastic parabolic differential equation}, 
      author={Fangfang Dou and Peimin Lü and Yu Wang},
      year={2024},
      eprint={2308.15741},
      archivePrefix={arXiv},
      primaryClass={math.NA},
      url={https://arxiv.org/abs/2308.15741}, 
}

@article{LZ:2024,
    AUTHOR = {L\"u, Qi and Zhang, Xu},
     TITLE = {Inverse problems for stochastic partial differential
              equations: some progresses and open problems},
   JOURNAL = {Numer. Algebra Control Optim.},
  FJOURNAL = {Numerical Algebra, Control and Optimization},
    VOLUME = {14},
      YEAR = {2024},
    NUMBER = {2},
     PAGES = {227--272},
      ISSN = {2155-3289,2155-3297},
   MRCLASS = {60H15 (35L04 35R30 35R60)},
  MRNUMBER = {4729334},
       DOI = {10.3934/naco.2023014},
       URL = {https://doi.org/10.3934/naco.2023014},
}

@article {WWW:2024,
    AUTHOR = {Wu, Bin and Wang, Ying and Wang, Zewen},
     TITLE = {Carleman estimates for space semi-discrete approximations of
              one-dimensional stochastic parabolic equation and its
              applications},
   JOURNAL = {Inverse Problems},
  FJOURNAL = {Inverse Problems. An International Journal on the Theory and
              Practice of Inverse Problems, Inverse Methods and Computerized
              Inversion of Data},
    VOLUME = {40},
      YEAR = {2024},
    NUMBER = {11},
     PAGES = {Paper No. 115003, 33},
      ISSN = {0266-5611,1361-6420},
   MRCLASS = {65C30 (35K20 35R30 35R60 60H15)},
  MRNUMBER = {4814264},
}

@book{lu2021mathematical,
  title={Mathematical Control Theory for Stochastic Partial Differential Equations},
  author={Qi Lü and Xu Zhang},
  series={Probability Theory and Stochastic Modelling},
  volume={101},
  year={2021},
  publisher={Springer Nature Switzerland AG},
  doi={10.1007/978-3-030-82331-3},
  isbn={978-3-030-82330-6}
}

@article {LLRP:2025,
    AUTHOR = {Lecaros, Rodrigo and L\'opez-R\'ios, Juan and P\'erez, Ariel
              A.},
     TITLE = {Lipschitz stability and reconstruction in inverse problems for
              semi-discrete parabolic operators},
   JOURNAL = {Inverse Problems},
  FJOURNAL = {Inverse Problems. An International Journal on the Theory and
              Practice of Inverse Problems, Inverse Methods and Computerized
              Inversion of Data},
    VOLUME = {41},
      YEAR = {2025},
    NUMBER = {11},
     PAGES = {Paper No. 115012, 30},
      ISSN = {0266-5611,1361-6420},
   MRCLASS = {35R30 (35K10)},
  MRNUMBER = {4988309},
       DOI = {10.1088/1361-6420/ae1bcb},
       URL = {https://doi.org/10.1088/1361-6420/ae1bcb},
}

@article{LPP:2025,
    AUTHOR = {Lecaros, Rodrigo and P\'erez, Ariel A. and Prado, Manuel F.},
     TITLE = {Carleman {E}stimate for {S}emi-discrete {S}tochastic
              {P}arabolic {O}perators in {A}rbitrary {D}imension and
              {A}pplications to {C}ontrollability},
   JOURNAL = {Appl. Math. Optim.},
  FJOURNAL = {Applied Mathematics and Optimization},
    VOLUME = {93},
      YEAR = {2026},
    NUMBER = {1},
     PAGES = {Paper No. 12},
      ISSN = {0095-4616,1432-0606},
   MRCLASS = {93B05 (65M06 93B07 93C20)},
  MRNUMBER = {5006586},
       DOI = {10.1007/s00245-025-10364-1},
       URL = {https://doi.org/10.1007/s00245-025-10364-1},
}

@article {Emanuilov:1995,
    AUTHOR = {\`Emanuilov, O. Yu.},
     TITLE = {Controllability of parabolic equations},
   JOURNAL = {Mat. Sb.},
  FJOURNAL = {Matematicheski\u i\ Sbornik},
    VOLUME = {186},
      YEAR = {1995},
    NUMBER = {6},
     PAGES = {109--132},
      ISSN = {0368-8666,2305-2783},
   MRCLASS = {93B05 (35K55 93C20)},
  MRNUMBER = {1349016},
MRREVIEWER = {Sergei\ A.\ Ivanov},
       DOI = {10.1070/SM1995v186n06ABEH000047},
       URL = {https://doi.org/10.1070/SM1995v186n06ABEH000047},
}
\end{document}